\documentclass{article}
\usepackage[margin=1in]{geometry}
\usepackage{amsthm, amssymb}
\newtheorem{theorem}{Theorem}
\newtheorem{proposition}[theorem]{Proposition}%
\newtheorem{remark}{Remark}%
\newtheorem{lemma}{Lemma}
\theoremstyle{definition}%
\newtheorem{definition}{Definition}
\newtheorem{convention}{Convention}

\usepackage{subfig}
\usepackage{placeins}
\usepackage{natbib}
\usepackage{authblk}
\usepackage{tikz}
\usepackage{url}
\usetikzlibrary{arrows.meta}

\usepackage{mathtools}
\newcommand{\R}{\mathbb R}

\newcommand{\dgm}{\mathsf{dgm}}

\newcommand{\im}{\mathrm{Im}}

\begin{document}

\title{Through the Grapevine: Vineyard Distance as a Measure of Topological Dissimilarity}

\author[1]{Alvan Arulandu}
\affil[1]{Harvard University, Cambridge, MA}
\author[2]{Daniel Gottschalk}
\affil[2]{University of Southern California, Los Angeles, CA}
\author[3]{Thomas Payne}
\affil[3]{University of North Carolina at Greensboro, Greensboro, NC}
\author[4]{Alexander Richardson}
\affil[4]{Bowdoin College, Brunswick ME}
\author[3]{Thomas Weighill}

\maketitle

\begin{abstract}
We introduce a new measure of distance between datasets, based on vineyards from topological data analysis, which we call the vineyard distance. Vineyard distance measures the extent of topological change along an interpolation from one dataset to another, either along a pre-computed trajectory or via a straight-line homotopy. We demonstrate through theoretical results and experiments that vineyard distance is less sensitive than $L^p$ distance (which considers every single data value), but more sensitive than Wasserstein distance between persistence diagrams (which accounts only for shape and not location). This allows vineyard distance to reveal distinctions that the other two distance measures cannot. In our paper, we establish theoretical results for vineyard distance including as upper and lower bounds. We then demonstrate the usefulness of vineyard distance on real-world data through applications to geospatial data and to neural network training dynamics.
\end{abstract}
\section{Introduction}

We are interested in the very general problem of quantifying the difference between two datasets $X$ and $Y$. For example, $X$ and $Y$ might correspond to two different samples taken at different times, or they could correspond to two covariates between which there could be some dependence. We are particularly interested in the case when $X$ and $Y$ can be represented as functions $f$ and $g$ on a fixed domain $\Omega$. Among the many possibilities for measuring the discrepancy between $X$ and $Y$, two useful choices are:

\begin{itemize}
    \item[(a)] computing an \textbf{$L^p$ distance}, i.e.~$||f-g||_p = \left(\int_\Omega |f-g|^p \right)^{1/p}$
    \item[(b)] computing a \textbf{persistence distance}, i.e.~quantifying the shape of $f$ and $g$ by persistence diagrams and computing the Wasserstein distance between them (see Section~\ref{sec:bg_ph}).
\end{itemize}

Method (a) ignores the topology of $\Omega$, and is a very discriminatory measure: indeed, $||f-g||_p = 0$ only when $f$ and $g$ are the same up to measure $0$. Method (b), on the other hand, takes the shapes of $f$ and $g$ into account, but is far from injective: if $f$ and $g$ are mirror images (via a symmetry of $\Omega$) then they will have persistence distance zero. 

Our central proposal in this paper is a novel distance measure -- the vineyard distance -- which we claim naturally bridges this gap. It is more discriminatory than persistence distance since it is not invariant under symmetries of $\Omega$. However, it is less sensitive than $L^p$ distance since it depends mostly on the relationship between certain key features of $f$ and $g$ such as local minima and maxima. As the name suggests, vineyard distance is defined using a vineyard~\cite{cohen2006vines}, a one-parameter interpolation of persistence diagrams, from the persistence diagram of $f$ to the persistence diagram of $g$. Our framework allows this vineyard to either be a part of the given data, or induced by a straight-line homotopy. As a natural consequence of the way it is defined, vineyard distance behaves like a hybrid between $L^p$ and persistence distance. Our numerical experiments and real-world examples will demonstrate, however, that it is a richer intermediate than a linear combination of $L^p$ distance and persistence distance.

To substantiate our claim that vineyard distance provides more discriminatory features than $L^p$ distance and persistence distance, we provide theoretical bounds (Theorem~\ref{thm:vd_lower_bound} and~\ref{thm:vd_upper_bound}) bounding the (unweighted) vineyard distance below by persistence distance, and above by a simplicial $L^p$ distance (via a result of Skraba-Turner~\cite{skraba2020wasserstein}). Numerical results on Gaussian mixtures (Section~\ref{sec:gaussians}) reinforce these results.

We demonstrate the usefulness of vineyard distance through two proof of concept applications: analyzing geospatial data and studying neural network training dynamics. For geospatial data, we adopt the framework of~\cite{kauba2024topological} for demographic data and show how vineyard distance highlights richer features than both $L^p$ distance and persistence distance. We study the training dynamics of neural networks by tracking the topology of the decision boundary and demonstrate how vineyard distance distinguishes between qualitatively different learning trajectories on a simple example. 

The structure of the paper is as follows. We begin with background on persistent homology in Section~\ref{sec:bg}. We then carefully establish a theoretical foundation in Section~\ref{sec:vvtheory}, proving the equivalence between two different definitions of vineyard distance. In Section~\ref{sec:bounds} we prove theoretical lower and upper bounds for various scenarios and versions of vineyard distance. Section~\ref{sec:numeric} contains experiments on synthetic data that demonstrate the theoretical results and support our hypothesis that vineyard distance is a useful hybrid between $L^p$ and persistence distance. Section~\ref{sec:app-geospatial} contains our application to geospatial data, while Section~\ref{sec:app-nn} contains experiments on neural networks.

\subsection{Related Work}

\subsubsection{Stable Representations and Distances for Persistent Homology}

Metrics on the space of persistence diagrams have been studied extensively in the literature. They have been used to show that small input perturbations in the input yield small differences between the original and perturbed diagrams, demonstrating stability in persistence diagrams. The first such result appeared in \cite{cohen-steiner_stability_2007} for the bottleneck distance, which was extended to $p$-Wasserstein distance in~\cite{skraba2020wasserstein}. Vectorizations of persistence diagrams can be used to impose (pseudo)metrics on the space of persistence diagrams; established methods for vectorization include persistence images \citep{adams2017persistence}, persistence landscapes \citep{bubenik2015statistical}, and persistence curves~\cite{chung2022persistence}. One can also apply kernels to persistence diagrams, such as the persistence scale-space kernel \citep{reininghaus2015stable}, the persistence-weighted Gaussian kernel \citep{kusano2016persistence}, and sliced-Wasserstein kernels \citep{carriere2017sliced}. All of these distance or kernel measures ignore the functions underlying the persistence diagrams which can be an advantage in many applications since it makes the distance invariant under certain transformations of the data. By contrast, our vineyard distance incorporates the underlying spaces and functions, tracking not only the change in topological features but their relative location.

\subsubsection{Vineyards and families of persistence diagrams}
The stability of persistence diagrams \citep{cohen2005stability} implies the existence of continuous single-parameter families of persistence diagrams known as vineyards. Given a homotopy between two functions (or any other continuous single-parameter family of functions on the same complex), following any point through the stack of persistence diagrams as this single-parameter continuously varies yields a vine. Originally applied to protein folding trajectories, \citep{cohen2006vines} gave an efficient algorithm for computing vines and vineyards without recomputing the persistence diagram from scratch at each time. Since this seminal work on vineyards, several generalizations of vineyards have emerged including vineyard modules as an algebraic analog of vineyards for persistence modules  \citep{turner2023representing} and persistence diagram bundles as a multi-dimensional generalization of vineyards that include multi-parameter persistence modules. Unlike vineyards, both vineyard modules and persistence diagram bundles are not necessarily the union or sum of ``vines," but contain more complex structure. Recent work has also yielded algorithms for computing vineyards from zigzag persistence \citep{dey2023computing} and using periodicity in vineyards to understand monodromy \citep{chambers2025braiding}. Our proposed vineyard distance is a natural statistic of a vineyard's complexity, which can enable the study of ensembles of vineyards rather than just individual vineyards.

% \subsubsection{Statistics for Topological Data Analysis}
% Topological summaries and pairwise statistics such as persistence diagrams and distances between them can be used for as the basis for a statistical analysis that incorporates topological features. Topological inference is a field of active study with results in hypothesis testing \citep{robinson_hypothesis_2017}, Bayesian inference \citep{ranganathan2006bayesian}, and robustness \citep{chazal2018robust}. However, handling probability distributions over persistence diagrams has proven challenging. After the first formal description of measures over persistence diagrams \citep{mileyko2011probability}, recent work has formalized the notion of medians \citep{turner_medians_2019} and probabilstic Fr\'echet means \citep{munch_probabilistic_2015} for persistence diagrams. Regarding distributions over persistence diagrams arising from data, \citet{papamarkou_random_2022} gave a Markov chain Monte Carlo method for sampling random persistence diagrams similar to those produced from a dataset, and \citet{bobrowski_universal_2023,bobrowski2022universality} presented several conjectures regarding a universal probability law for persistence diagrams deriving from Vietoris-Rips complexes of random point clouds. 

% We frame vineyard distance as a useful topological statistic for data pairs and time-series within the larger toolbox of topological data analysis, and we propose vineyards arising from data as an object for future study.

\subsubsection{Fused distances}
As described in the introduction, vineyard distance can be thought of as a hybrid between an $L^p$ distance and a purely topological distance. In this sense, it is related to other distances which incorporate both global topology as well as information about the underlying data such as fused Gromov-Wasserstein distance~\cite{titouan2019optimal, vayer2020fused}, distances between decorated merge trees~\cite{curry2022decorated} and distances between enriched barcodes~\cite{cang2020persistent}. We note that vineyard distance inherits its hybrid nature as a natural consequence of the way it is defined, instead of as a linear combination of different losses. The numerical experiments in Section~\ref{sec:gaussians} make it clear that vineyard distance also behaves qualitatively differently than a simple linear combination of $L^p$ and topological distance.

\subsubsection{Geospatial data}
Topological data analysis has been employed to study spatial patterns in civic and health data. Prior work analyzes voting patterns \citep{feng2021persistent}, redistricting  \citep{Duchin2022}, and city-scale demographics via tract-level filtrations \citep{kauba2024topological}. Notably, Hickock et al.~\cite{Hickock2022} analyzed time-varying spatial anomalies in COVID-19 data via vineyards where each vine corresponds with an anomaly. In these applications where the underyling geospatial data is inherently temporal, vineyards can embed the changes of topological features over time with vineyard distance being the summary statistic. Our application in Section~\ref{sec:app-geospatial} shows, however, that even when no temporal structure exists, the vineyard distance of a straight-line homotopy between scalar functions can provide unique insights about the relationship between different datasets.

\subsubsection{Deep learning}

Persistent homology has been used to probe neural network training dynamics by tracking topology changes across layers \citep{naitzat2020topology}, complexity measures like neural persistence \citep{rieck2018neural}, and topology-aware model evaluation via topological parallax \citep{smith2023topological}. This literature focuses largely on understanding the state of the neural network after training is completed (or perhaps at a single point during training). We propose that vineyard distance can provide insights into qualitative aspects of neural network training through our proof of concept in Section~\ref{sec:app-nn}.

\section{Background}\label{sec:bg}

\subsection{Persistent homology}\label{sec:bg_ph}

Persistent homology analyzes topological properties (e.g.~connected components, holes, voids etc) within data at various scales or thresholds. Rather than work from a static picture, the fundamental input to persistent homology is an increasing sequence of spaces. A standard way to produce such a sequence is by starting with a simplicial complex and a function $f$ defined on the simplices.

By a \textbf{(finite) simplicial complex} $\mathcal{K}$, we mean a set of vertices $V(\mathcal{K})$ and a set of simplices $\Sigma(\mathcal{K}) \subseteq 2^{V(\mathcal{K})}$. While every simplicial complex has many geometric realizations, our analysis will depend only on the combinatorial structure of the pair $(V(\mathcal{K}), \Sigma(\mathcal{K}))$. Given two simplices $\sigma, \tau \in \Sigma(\mathcal{K})$, we say $\sigma$ is a \textbf{face} of $\tau$ and $\tau$ is a \textbf{coface} of $\sigma$ if $\sigma \subseteq \tau$. 

\begin{convention}
Let $\mathcal{K}$ be a simplicial complex. A function $f: \mathcal{K} \to \R$ will always mean a function from the simplices of $\mathcal{K}$ to $\R$ such that $f(\sigma) \leq f(\tau)$ whenever $\sigma$ is a face of $\tau$. Such functions are sometimes called \textbf{monotone} functions. The set of all (monotone) functions $\mathcal{K} \to \mathbb{R}$ will be denoted by $\mathbb{R}^\mathcal{K}$ and given the product topology.
\end{convention}

Note that the definition of a (monotone) function is very different to that of a continuous function on $\mathcal{K}$ with some topology. A function $f: \mathcal{K} \to \R$ allows us to define, for any $t \in \R$, the \textbf{sublevel set} $\mathcal{K}_t$ as the subcomplex of $\mathcal{K}$ consisting of all simplices $\sigma$ with $f(\sigma) \leq t$. Assuming that the points in the image of $f$ are given by $t_0 \leq t_1 \leq \cdots\leq t_{i}\leq t_{i+1}\leq \cdots $, we obtain a sequence of increasing simplicial complexes

$$\mathcal{K}_{t_0}\subseteq \mathcal{K}_{t_1}\subseteq\cdots\subseteq \mathcal{K}_{t_{i}}\subseteq \mathcal{K}_{t_{i+1}}\subseteq\cdots$$
which we call a  \textbf{filtered simplicial complex} or \textbf{filtration}.

The inclusion maps $\mathcal{K}_{t_i} \to \mathcal{K}_{t_{i+1}}$ give rise to maps on simplicial homology in each dimension $d$:
$$
H_d(\mathcal{K}_{t_0}) \to H_d(\mathcal{K}_{t_1}) \to \cdots\to H_d(\mathcal{K}_{t_i}) \to H_d(\mathcal{K}_{t_{i+1}}) \to \cdots 
$$
Here, the homology groups $H_d(\mathcal{K}_{t_i})$ are vector spaces (when using coefficients from a field, as is standard in practical applications) whose dimension corresponds to the number of topological features of various kinds. For example, $\dim H_0(\mathcal{K}_t)$ is the number of connected components in $\mathcal{K}_t$, while $\dim H_1(\mathcal{K}_t)$ is the number of holes. The maps between the homology groups encode the evolution of these features over the length of the sequence. Each feature is \textbf{born} (appears) and may \textbf{die} (disappear or is merged into older features). Slightly more formally, the Fundamental Theorem of Persistent Homology~\cite{zomorodian2004computing} guarantees the existence of a basis for each $H_d(\mathcal{K}_{t_i})$ such that each basis element has a well-defined first and last appearance in the sequence; see for example~\cite{otter2017roadmap} for more concrete details.

A standard summary of these births and deaths is given by a \textbf{persistence diagram}, which records each birth/death pair as a point $(b,d)$ in the plane. We denote the persistence diagram associated to the function $f$ in dimension $d$ by $\dgm_d(f)$. It will be convenient to have a formal definition of a \textbf{persistence diagram} that does not refer to a specific simplicial function that produces it.  

\begin{definition}
Let $\mathbb{R}^2_{x \leq y} = \{(x,y) \in \mathbb{R} \mid x \leq y\}$. A \textbf{persistence diagram} is a finite multiset $P$ of points in $\mathbb{R}^2_{x \leq y}$. We denote the set of all persistence diagrams by $\mathcal{D}$.
\end{definition}

We use the term ``multiset'' to include cases where points in a diagram have multiplicity greater than $1$ (a formal definition and notation can be found in Section 2.1 of ~\cite{pritchard2024coarse}, for example). We will refer to the subset $\{(x,x) \in \mathbb{R}^2_{x\leq y}\}$ as the \textbf{diagonal} and denote it by $\Delta$. Note that by our definition persistence diagrams can include points on the diagonal, but these will be thought of as negligible. 

\begin{remark}
    Various definitions of persistence diagram exist in the literature, designed to suit a particular application or make certain definitions or proofs more efficient. For example, some definitions include an infinite supply of points on the diagonal, i.e.~points of the form $(x,x)$. Other definitions allow values of infinity for death values, which correspond to features which are present at the end of the filtration. In practice, infinite values are often removed or converted to a suitable ceiling value. 
\end{remark}

Persistence diagrams contain relevant topological information for a given function $f$ and complex $\mathcal{K}$. By defining distance measures with respect to diagrams, we can build notions of topological distance for two given functions $f,g:\mathcal{K}\to \R$ on the same complex (or even on different complexes). The Wasserstein distance is the most prominent such example. By a \textbf{partial bijection} from a (multi)set $P$ to a (multi)set $Q$, we mean a bijection $\pi: A \to B$ where $A$ and $B$ are sub(multi)sets of $P$ and $Q$.

\begin{definition}
    Given two persistence diagrams $P$ and $Q$, and $1\leq p,q\leq \infty$, we define the $(p,q)$-Wasserstein distance $W_{p,q}(A,B)=\min_{\pi:A\to B}D(\pi)$ where $\pi$ ranges over partial bijections from $P$ to $Q$ and
    $$D(\pi)=\left(\sum_{\alpha\in A}d_p(\alpha,\pi(\alpha))^q+\sum_{\alpha\not\in A}d_p(\alpha,\Delta)^q+\sum_{\beta\notin B}d_p(\beta,\Delta)^q\right)^{1/q}$$
    with $d_p$ being the $L^p$ metric for $\R^2$. 
\end{definition}
Informally, the Wasserstein distance computes the cost of a best matching between two diagrams, except that we allow any number of points in either diagram to match to the diagonal. We commonly refer to $p$ and $q$ as the internal and external metric coefficient respectively, adopting the shorthand $W_{p}=W_{p,p}$. Notice that when $p=q=\infty$, we get
    $$W_{\infty}=\min_{\pi:A\to B} \max\left(\max_{\alpha\in A}d_{\infty}(\alpha,\pi(\alpha)), \max_{\alpha\not\in A}d_{\infty}(\alpha,\Delta), \max_{\beta\not\in B}d_{\infty}(\beta,\Delta)\right)$$
    which is equivalent to the \textbf{bottleneck distance}. 
    The following important result states that persistence diagrams are stable with respect to the input function $f$ when measured by bottleneck distance.

\begin{theorem}\label{them:cohensteiner}[Cohen-Steiner et al.~\cite{cohen-steiner_stability_2007}]
        Let $f$ and $g$ be functions on a finite simplicial complex $\mathcal{K}$ and let $\dgm_d(f)$ and $\dgm_d(g)$ be the associated sublevel set filtration persistence diagrams in some dimension $d$. Then,
        $$W_{\infty}(\dgm_d(f),\dgm_d(g))\leq \lVert f-g\rVert_{\infty}$$
    \end{theorem}

\subsection{Vertex-based sublevel set filtrations} \label{sec:vertexbased}

We now turn our attention to some ways to construct a function $f: \mathcal{K} \to \R$ from real data. First, we consider the case when our data can be encoded as a simplicial complex $\mathcal{K}$ (e.g.~a graph or triangulation) and a function  $f: V(\mathcal{K}) \to \mathbb{R}$ from the vertices of $\mathcal{K}$ to $\mathbb{R}$ (e.g.~node weights). We can extend $f$ to all simplicies by defining
$$
f(\sigma) = \max \{f(v) \mid v \in \sigma \}
$$
to give a filtration and therefore a persistence diagram (this is sometimes called a \textbf{lower-star filtration}). Note that this definition amounts to forcing that a simplex from $\mathcal{K}$ appear in $\mathcal{K}_t$ if and only if all its vertices are present. Abusing notation, we write $\dgm_d(f)$ to mean the $d$-dimensional persistence diagram for this filtration.

We should mention two important variants of the construction described above. Firstly, we can always ``filter down on $f$'' by replacing $f$ by $-f$ (or $M - f$ for a suitable constant $M$). This could reveal new information about $f$; for example, local maxima may be detected instead of local minima. Secondly, the construction can easily be adapted from simplicial complexes to cubical complexes, which provide a natural setting for image analysis and are more computationally efficient in some cases~\cite{wagner2011efficient}.

\subsection{Vietoris-Rips complexes}

In some cases, our data consists of a finite point cloud $X \subseteq \mathbb{R}^n$ (or more generally, a finite metric space). In that case, we can define a filtration using the \textbf{Vietoris-Rips complex} construction. For $t > 0$, we define $\mathcal{K}_t$ to be the complex whose vertices are the points in $X$ and where a subset $\sigma \subseteq X$ spans a simplex if and only if the diameter of $\sigma$ is at most $t$. Since the set of all distances realized in $X$ is finite, this complex changes at finitely many $t$-values, giving rise to a filtration 

$$\mathcal{K}_{t_0}\subseteq \mathcal{K}_{t_1}\subseteq\cdots\subseteq \mathcal{K}_{t_{i}}\subseteq \mathcal{K}_{t_{i+1}}\subseteq\cdots$$
which as before produces a persistence diagram.

To align this construction with the general one given in Section~\ref{sec:bg_ph}, note that we can consider the maximal simplicial complex $\mathcal{K}$ whose vertex set is $X$ and define the function $f: \mathcal{K} \to \R$ by $f(\sigma) = \mathsf{diam} (\sigma)$ to obtain the filtration above. This perspective eventually leads to the following corollary of Proposition~\ref{them:cohensteiner} proved in~\cite{chazal2009gromov}.

\begin{proposition}\label{prop:pd_stability_vr}[Chazal et al. \cite{chazal2009gromov}]
        Let $X$ and $Y$ be finite point clouds (or finite metric spaces), and let $\dgm_d(f)$ and $\dgm_d(g)$ be the associated Vietoris-Rips filtration persistence diagrams in some dimension $d$. Then,
        $$W_{\infty}(\dgm_d(f),\dgm_d(g))\leq d_{GH}(X, Y)$$
         where $d_{GH}$ is the Gromov-Hausdorff distance.
    \end{proposition}
In particular, if the points of a point cloud $X$ are moved by a bounded amount, the Vietoris-Rips persistence diagram also changes only by a bounded amount (in bottleneck distance).

\section{Vineyards and vineyard distance}\label{sec:vvtheory}

\subsection{Vineyards}

Vineyards can be thought of as time-varying persistence diagrams. 

\begin{definition}
        Let $\mathcal{D}$ denote the set of all persistence diagrams endowed with the topology arising from bottleneck distance (or equivalently, any $W_{p,q}$ distance). A \textbf{vineyard} is a continuous map $\mathcal{V}:[0,1] \to \mathcal{D}$. 
\end{definition}

In most cases, vineyards arise from a time-varying function on a fixed simplicial complex. Given two functions $f,g:\mathcal{K}\to \R$ on the same complex, a \textbf{homotopy from $f$ to $g$} is a continuous map $h:I \to \R^{\mathcal{K}}$ where $I = [0,1]$ such that $h(0)=f$ and $h(1)=g$.

    \begin{definition}
        Given a simplicial complex $\mathcal{K}$ and a homotopy $h: I \to \R^{\mathcal{K}}$, we define the \textbf{dimension $d$ vineyard induced by $h$} to be the map $\mathcal{V}_h:I\to \mathcal{D}$ defined by $\mathcal{V}_h(t) = \dgm_d(h(t))$. 
    \end{definition}
    
Vineyards were introduced in \cite{cohen2006vines}, where an algorithm is given for computing the vineyard induced by a homotopy by dynamically updating the persistence diagram, resulting in a more efficient method than recomputing the persistence at each time step. A vineyard $\mathcal{V}$ arising from such an algorithm comes with extra data, namely a matching between the persistence diagrams $\mathcal{V}(s)$ and $\mathcal{V}(t)$ when $s$ and $t$ are sufficiently close. A general vineyard as defined above, however, need not come already equipped with such matchings, and may not have a unique choice for them either. A formal algebraic treatment of the extra data carried by vineyards induced by homotopies has been initiated in \cite{turner2023representing}, via the notion of vineyard modules. 

As $t$ varies, the points in the diagram $\mathcal{V}(t)$ trace out curves in the plane, which are commonly called vines. That every vineyard (perhaps with mild conditions) can be decomposed into finitely many vines is stated in various places in the literature without proof. In order to define our main object of interest in this paper, we will require a stronger condition, namely that these vines have lengths that can be measured, so we make this an assumption for the remainder of the paper.

\begin{convention}\label{con:vines}
    We assume that any vineyard $\mathcal{V}$ in this paper can be decomposed into finitely many rectifiable \textbf{vines} $(V_i: [0,1] \to \mathbb{R}^2)_{1 \leq i \leq m}$ such that for each $t$, $\mathcal{V}(t)$ is the multiset $\{V_i(t) \mid 1 \leq i \leq m\}$. Moreover, we assume that within this decomposition, each vine $V$ falls into one of four classes:
    \begin{itemize}
        \item $V^{**}$: vines which never intersect the diagonal,
        \item $V^{*\circ}$: vines which end on the diagonal, that is, such that $\im(V) \cap \Delta =\{x\}$, and $V^{-1}(x)$ is a closed interval containing $1$. The minimum of this closed interval will be called the \textbf{end time} of the vine. 
        \item $V^{\circ*}$: vines which start on the diagonal, that is, such that $\im(V) \cap \Delta =\{x\}$ and $V^{-1}(x)$ is a closed interval containing $0$. The maximum of this closed interval will be called the \textbf{start time} of the vine.
        \item $V^{\circ \circ}$: vines which start and end on the diagonal, that is, such that $\im(V) \cap \Delta =\{x,y\}$ and $V^{-1}(x) \cup V^{-1}(y)$ is a union of two closed intervals which contains $0$ and $1$. Start and end times are defined similarly to the above.
    \end{itemize}
\end{convention}

The main difference between Convention~\ref{con:vines} and the classification of vines in Section 4 of \cite{cohen2006vines} is that we include the start and/or end points of each vine in the last three categories as part of the vine. This will make the theoretical results easier to state and prove since each vine has domain $[0,1]$, but it is not an important restriction since any ``open'' vine $(a,b) \to \mathbb{R}^2_{x\leq y}$ can be extended by adding the end points.

The regularity conditions in Convention \ref{con:vines} are motivated by the need to have a finite sum of well-defined line integrals of vines in Definition~\ref{def:vineyard}. However, these conditions hold for all vineyards of practical interest that we are aware of. The authors of the original vineyard paper note~\cite[Section 4]{cohen2006vines} that in many applications vineyards are not continuous but are instead linear interpolations from a discrete set of observations. In such cases the restrictions in Convention~\ref{con:vines} are easily verified to hold. The other main case of interest for this paper are straight-line homotopies. As noted by the same authors, vineyards arising from smooth homotopies have a decomposition into piecewise smooth vines which fall into the classes listed in Convention~\ref{con:vines}. 

In general, there may be more than one vine decomposition of a vineyard since vines can collide. Note that we say that two vines $V_1$ and $V_2$ \textbf{collide at $t$} if $V_1(t) = V_2(t)$, to distinguish from the strictly more general phenomenon wherein the images of $V_1$ and $V_2$ cross. A collision is only possible when some diagram $\mathcal{V}(t)$ contains points of multiplicity greater than $1$. In practical applications this is usually assumed to be a rare occurrence.

\subsection{Vineyard distance}

We are now ready to introduce the main object of study in the present paper.

\begin{definition}\label{def:weighting}
By a \textbf{weighting function}, we mean a uniformly continuous function $w: \mathbb{R}_{x\leq y}^2 \to \mathbb{R}_{\geq 0}$ which is non-zero outside of the diagonal.
\end{definition}

    \begin{definition}\label{def:vineyard}
    Let $\mathcal{V}$ be a vineyard.
        Fix any decomposition of $\mathcal{V}$ into a set of vines which we abuse notation to denote by $\mathcal{V}$ as well. Let $w$ be a weighting function. The $w$-\textbf{weighted vineyard distance} $\mathbb{V}^{(w)}(\mathcal{V})$ of this vineyard is
        \begin{equation}\label{eq:distancebyvine}
           \mathbb{V}^{(w)}(\mathcal{V}) = \sum_{V \in \mathcal{V}} \int_V w ds
        \end{equation}
        where the $\ell^\infty$ metric on the plane is used to compute the line integrals.
        \end{definition}   

        We will soon see that the vineyard distance is independent of the decomposition into vines (Theorem~\ref{thm:distancebyW}). We can formulate the vineyard distance as the limit of sums of arc elements as follows:

    \begin{equation}\label{eq:arcelements}
        \mathbb{V}^{(w)} = \sum_{V\in \mathcal{V}} \lim_{n\to\infty}\sum_{i=1}^{n}\left\lVert V\left(\frac{i}{n}\right)-V\left(\frac{i-1}{n}\right)\right\rVert_{\infty} w\left(V\left(\frac{i-1}{n}\right)\right)
    \end{equation}
    where $||\cdot ||_\infty$ denotes the sup-norm $||\mathbf{v}||_\infty = \max_i \mathbf{v}_i$. The purpose of the weighting function $w$ is to down-weight long vines near the diagonal (typically corresponding to features created by small perturbations or noise) in certain applications. We briefly define two typical choices for this weighting function.

    \begin{definition}
        The \textbf{standard weighting function} is the function $w((x,y)) = \Delta((x,y))$ where $\Delta(p)$ denotes the Euclidean distance from $p$ to the diagonal. The \textbf{uniform weighting function} is $w(x,y) = 1$. When $w$ is the uniform weighting function, we will refer to $\mathbb{V}^{(w)}(\mathcal{V})$ as the \textbf{unweighted vineyard distance} and denote it by $\mathbb{V}(\mathcal{V})$.
    \end{definition}

\begin{remark}
Historically, a vineyard $\mathcal{V}:[0,1] \to \mathcal{D}$ is often drawn as a 3D plot where the $z$-axis denotes the domain (thought of as the time variable) and each horizontal slice contains a persistence diagram (see e.g.~\cite{Hickock2022}), omitting points on the diagonal. Our vineyard distance is computed from the projection of these 3D curves to the $x$-$y$ plane. A natural question is how vineyard distance compares to the length of the original 3D curves. Depending on how the length of the 3D curves is measured, there may be only a weak relationship. The most interesting metric on $\mathbb{R}^2_{x\leq y} \times [0,1]$ for the purposes of this comparison is the metric 
$$d((x,y,t), (x',y',t')) = \max(|x-x'|, |y-y'|) + |t-t'|$$
If the vines never touch the diagonal and no weighting is used then the length of the 3D curves with this metric is just the vinyeard distance plus the number of vines. However, if weighting is used, or time on the diagonal is not counted towards the length of 3D curves,  then the length of the 3D curves can depend on the parameterization of the vineyard. By contrast, the vineyard distance is independent of the parameterization of $\mathcal{V}$. In summary, while 3D length is another natural choice for defining vineyard distance, our choice leads to less parameter dependence especially with weighting.
\end{remark}

A common way to produce a vineyard, and hence compute a vineyard distance, between two functions on the same simplicial complex $\mathcal{K}$, is using a straight-line homotopy. Note that any reparameterization of the straight line homotopy also gives the same vineyard distance.

\begin{definition} \label{def:straight-line-homotopy}
    Given two functions $f$ and $g$ on a set $X$, the \textbf{straight-line homotopy} is defined by
\[
h_t(x) =  (1-t)\cdot f(x) + t\cdot g(x)
\] 
If $f$ and $g$ are functions on the same simplicial complex $\mathcal{K}$ or on its vertices, then the  \textbf{straight-line vineyard}, denote by $\mathcal{V}_d(f,g)$, is the vineyard $
\mathcal{V}(t) = \dgm_d(h_t)
$
and we denote the corresponding vineyard distance by $\mathbb{V}_d(f,g)$. 
\end{definition}

% \begin{remark}
%     When $f$ and $g$ are functions on the vertices of a simplicial complex $\mathcal{K}$, there are in fact two natural choices of homotopy between the functions induced by $f$ and $g$ on $\mathcal{K}$ (in the sense of Section~\ref{sec:vertexbased}). One option is to apply Definition~\ref{def:straight-line-homotopy} to the induced functions. Alternately, one could first interpolate between $f$ and $g$ on vertices, and then induce functions on the simplices at each time step. The choice might depend on the application, but to apply the theoretical results in Section~\ref{sec:bounds}, the former definition is most convenient.
% \end{remark}

\subsection{Formulation in terms of weighted Wasserstein distance}

In this section, we show that the vineyard distance can be computed without a vine decomposition using infinitesimal Wasserstein distances, possibly adapted to include weights. We begin by introducing a weighted version of the Wasserstein distance between persistence diagrams. Note that for the uniform weighting, the weighted and unweighted versions coincide, that is, $W_{p,q}^{(w)} = W_{p,q}$. 

\begin{definition}
    Let $P$ and $Q$ be two persistence diagrams, and let $w$ be a weighting function. Then, the \textbf{$w$-weighted $(p,q)$-Wasserstein distance} is $W_{p,q}^{(w)}(P, Q)=\min_{\pi:A\to B}D^{(w)}(\pi)$ where
        $$
    D^{(w)}(\pi)=\left(\sum_{\alpha\in A} d_{p}^{(w)} (\alpha,\pi(\alpha))^q+
    \sum_{\alpha\notin A} d_{p}^{(w)}(\alpha,\Delta)^q+
    \sum_{\beta\notin B} d_{p}^{(w)}(\beta,\Delta)
    )^q\right)^{1/q}
$$
    and

    $$d_p^{(w)}(p_0,p_1) = d_p(p_0, p_1)\frac{w\left(p_0\right)+ w(p_1)}{2}
    $$
    We call any $\pi$ realizing the minimum above an \textbf{optimal matching} for $W_{p,q}^{(w)}(P, Q)$.
\end{definition}

\begin{lemma}\label{lem:closeW}
    Let $P$ and $Q$ be persistence diagrams with no points on the diagonal, and let $\hat{P}$ and $\hat{Q}$ denote the sets (that is, points without multiplicity) of points in $P$ and $Q$. Let $w$ be a weighting function. Let
    \begin{align*}
        \varepsilon_w & = \min \left(\frac{1}{2} \min_{p \in \hat{P} \cup \hat{Q}} w(p),\ 1 \right) \\ \varepsilon_p & = \min_{p, p' \in \hat{P},\ p\neq p'} d_\infty(p,p') \\
        \varepsilon_q & = \min_{q, q' \in \hat{Q},\ q\neq q'} d_\infty(q,q') \\
        \varepsilon_\Delta & = \min_{p \in \hat{P} \cup \hat{Q}}  d_\infty(p, \Delta) 
        \end{align*}
    Suppose that $W_{\infty}(P, Q) < \varepsilon = \frac{1}{2}\varepsilon_w \cdot \min(\varepsilon_p, \varepsilon_q, \varepsilon_\Delta)$. Then, for each point $p\in \hat{P}$, there is a unique point $q \in \hat{Q}$ such that $d(p,q) < \varepsilon$. Further, this unique choice of points gives rise to an optimal matching for both the weighted Wasserstein distance $W^{(w)}_{\infty, 1}$ and the bottleneck distance.     
\end{lemma}
\begin{proof}
First, note that if two points $q, q' \in \hat{Q}$ are within $\varepsilon$ of a point $p \in \hat{P}$, then $d_\infty(q,q') < \varepsilon_p$, so $q = q'$. Similarly, if $p, p' \in \hat{P}$ are within $\varepsilon$ of a point $q \in \hat{Q}$, then $p = p'$. Consider an optimal matching $\pi$ from $P$ to $Q$ for the bottleneck distance. First, note that this matching cannot match any off-diagonal points to the diagonal, since otherwise the cost would exceed $\varepsilon_\Delta$. Thus, $\pi$ is a surjection from $\hat{P}$ to $\hat{Q}$. Each point $p \in P$ must be matched to a point $q$ with $d(p,q) < \varepsilon$ or else the cost of $\pi$ will exceed $\varepsilon$. It follows that $\pi$ is also injective on $\hat{P}$. A similar argument shows the same for an optimal matching for $W^{(w)}_{\infty, 1}$.
\end{proof}

\begin{theorem}\label{thm:distancebyW}
    Let $\mathcal{V}: [0,1] \to \mathcal{D}$ be a vineyard, and let $w$ be a weighting function. Then
    \begin{equation}\label{eq:distancebyW}
    \mathbb{V}(\mathcal{V}) = \lim_{n\to\infty}\sum_{i=1}^{n}W_{\infty,1}^{(w)}\left(\mathcal{V}\left(\frac{i}{n}\right),\mathcal{V}\left(\frac{i-1}{n}\right)\right)
    \end{equation}
\end{theorem}
\begin{proof}
We first show that we can reduce to a simpler case when the vines are piecewise linear. For any $n > 0$, let $V^{(n)}$ be the polygonal path approximating each vine $V$ with vertices $(V(i/n))_{0\leq i \leq n}$. Since $V^{(n)}$ agrees with $V$ at these vertices, the right hand side of (\ref{eq:distancebyW}) does not change if $V$ is replaced by $V^{(n)}$. On the other hand, the line integral of $V^{(n)}$ converges to the line integral of $V$ as $n \to \infty$. It is therefore enough to prove (\ref{eq:distancebyW}) for vineyards with piecewise linear vines.

Suppose then that all vines have this form. For every pair of vines, the set of times $t$ for which they coincide is the union of a finite set of disjoint closed intervals and isolated points. Putting together all such isolated points, the endpoints of all such closed intervals, and the start and end times of all vines, we can assemble a finite set of points. We can always find a subset $A \subseteq (0,1)$ consisting of finitely many disjoint closed intervals $(C_i)_{1\leq i \leq m}$ which avoid this set and satisfy the property that $
\sum_{V \in \mathcal{V}} \sum_{i=1}^m \int_{V |_{C_i}} w ds
$ is arbitrarily close to the vineyard distance for $\mathcal{V}$. 

Note that on each of these closed intervals vines either completely coincide or completely avoid one another. Each vine also avoids the diagonal or remains at a single point on the diagonal throughout the interval. It follows that the values
\begin{align*}
\varepsilon_p & =  \min \{||V(t) - V'(t)||_\infty \mid V, V' \in \mathcal{V},\ t \in A,\ V(t) \neq V(t') \} \\
\varepsilon_\Delta & = \min \{ ||V(t)-\Delta||_\infty \mid V \in \mathcal{V},\ t\in A,\ V(t) \notin \Delta \} 
\end{align*}
exist and are positive, as is the value
$$
\varepsilon_w = \min \{ w(V(t)) \mid V \in \mathcal{V},\ t \in A,\ V(t) \notin \Delta \},
$$
since $w$ is nonzero outside the diagonal. 

We can now pick $n$ so large that if $|t - t'| \leq 1/n$, then 
\begin{align*}
     W_{\infty}(\mathcal{V}(t), \mathcal{V}(t')) & < \frac{1}{2} \varepsilon_w \min(\varepsilon_p, \varepsilon_\Delta), \\
     \forall_{V \in \mathcal{V}} ||\mathcal{V}(t)-\mathcal{V}(t')||_\infty & < \frac{1}{2} \varepsilon_w \min(\varepsilon_p, \varepsilon_\Delta),
\end{align*}
and $t$ and $t'$ must be in the same $C_i$. Consider two diagrams $\mathcal{V}(t)$ and $\mathcal{V}(t')$ where $|t - t'| = 1/n$. By Lemma \ref{lem:closeW}, the matching given by the vines, i.e.~matching $V(t)$ to $V(t')$ for each $V \in \mathcal{V}$, is an optimal matching for $W^{(w)}_{\infty,1}$. It follows that 
$$
W^{(w)}_{\infty,1}(\mathcal{V}(t), \mathcal{V}(t')) = \sum_{V \in \mathcal{V}} \left( ||V(t) - V(t')||_\infty \cdot \frac{w(V(t))+w(V(t'))}{2} \right)
$$
and thus
$$
\sum_{i\in J}W_{\infty,1}^{(w)}\left(V\left(\frac{i}{n}\right),V\left(\frac{i-1}{n}\right)\right) = \sum_{V\in \mathcal{V}} \sum_{i \in J} \left(\left\lVert V\left(\frac{i}{n}\right)-V\left(\frac{i-1}{n}\right)\right\rVert_{\infty}  \cdot \frac{w(V\left(\frac{i}{n}\right))+w(V\left(\frac{i-1}{n}\right))}{2} \right)
$$
where $J$ is the set of $i$ such that $\{i/n, (i-1)/n\} \subseteq A$. As $n \to \infty$, the size of $A$ grows, and these sums converge to produce equation (\ref{eq:distancebyW}) as required.
\end{proof}

\section{Theoretical bounds on vineyard distance}\label{sec:bounds}

\subsection{Bounds for unweighted vineyard distance}

We now establish some bounds on vineyard distance in terms of persistence distance and $L^p$ distance.

\begin{theorem}\label{thm:vd_lower_bound}
Let $\mathcal{V}$ be a vineyard, and denote $P = \mathcal{V}(0)$ and $Q = \mathcal{V}(1)$. Then the unweighted vineyard distance $\mathbb{V}(\mathcal{V})$ satisfies
$$
W_{\infty, 1}(P, Q) \leq  \mathbb{V}(\mathcal{V}) 
$$
\end{theorem}
\begin{proof}
    Since $W_{\infty, 1}$ is a metric, this result follows from the triangle inequality and Theorem \ref{thm:distancebyW}.
\end{proof}

If the vineyard distance comes from a straight-line homotopy then we can derive an upper bound as well. Recall the following generalization of Theorem~\ref{them:cohensteiner}.

\begin{theorem}\label{thm:skrabaturner}[Theorem 4.6 in \cite{skraba2020wasserstein}]
    Let $f$ and $g$ be functions on a finite simplicial complex $\mathcal{K}$ and let $\dgm_d{f}$ and $\dgm_d{g}$ be the associated sublevel set filtration persistence diagrams in some dimension $d$. Then,
    \begin{align*}
        W_{p}(\dgm_d(f), \dgm_d(g))^p \leq \sum_{\dim(\sigma) \in \{d, d+1\} }|f(\sigma) - g(\sigma)|^p
    \end{align*}
\end{theorem}

\begin{theorem}\label{thm:vd_upper_bound}
Let $f$ and $g$ be functions on the simplicial complex $\mathcal{K}$, and let $\mathcal{V}$ be the vineyard in dimension $d$ induced by the straight-line homotopy between $f$ and $g$. Then, the unweighted vineyard distance satisfies
$$
\mathbb{V}(\mathcal{V}) \leq  \sum_{\dim(\sigma) \in \{d, d+1\} }|f(\sigma) - g(\sigma)|
$$
\end{theorem}
\begin{proof}
Note that any optimal matching for the $W_{1,1}$ distance gives a matching (not necessarily optimal) for $W_{\infty, 1}$ with lower cost since $d_\infty \leq d_1$. It follows that $W_{\infty,1} \leq W_{1, 1}$. Let $h$ be the straight-line homotopy betwen $f$ and $g$. For any $t, t' \in [0,1]$, we have that 
    $$
    W_{\infty, 1}(\mathcal{V}(t), \mathcal{V}(t')) \leq W_{1, 1}(\mathcal{V}(t), \mathcal{V}(t')) \leq \sum_{\dim(\sigma) \in \{d, d+1\} }|h(t)(\sigma) - h(t')(\sigma)| 
    $$
by Theorem~\ref{thm:skrabaturner}. From Theorem~\ref{thm:distancebyW}, we thus have that
$$
\mathbb{V}(\mathcal{V}) \leq  \lim_{n\to \infty} \sum_{i=1}^n \ \sum_{\dim(\sigma) \in \{d, d+1\} }\left\vert h\left(\frac{i}{n}\right)(\sigma) - h\left(\frac{i-1}{n}\right)(\sigma) \right\vert
$$
However, since $h$ is a straight-line homotopy, each term in the outer summation is identical. This yields $\sum_{\dim(\sigma) \in \{d, d+1\} }|f(\sigma) - g(\sigma)|$ on the right hand side, which gives the result.
\end{proof}

In summary, when using the straight-line homotopy, the vineyard distance is bounded below by the persistence distance between the endpoints and bounded above by an appropriately defined $L^1$ distance. We thus expect our distance to have more power to distinguish different functions than a classical Wasserstein distance between persistence diagrams, while also being less sensitive to certain kinds of deformations than a topology-blind $L^p$ distance.

\subsection{Lower bound for the standard weighting}

Theorems~\ref{thm:vd_lower_bound} and \ref{thm:vd_upper_bound} can be adapted to the weighted setting via multiplicative constant factors corresponding to $\min(w)$ and $\max(w)$ where appropriate. However, for the standard weighting this is not helpful as $w$ has a minimum of $0$ and no maximum. We thus are interested in finding bounds for the standard weighting. To do so requires us to understand optimal paths (as measured by line integrals) between points in $\mathbb{R}^2_{x \leq y}$.

\begin{definition}
    Let $w$ be a weighting function and let $p,q \in \mathbb{R}^2_{x\leq y}$. The \textbf{$w$-path distance} between $p$ and $q$ is
    $$
    d_{\text{path}}^{(w)}(p,q)  = \min_{\gamma \in \Gamma(p,q)} \int_\gamma w ds
    $$
    where $\Gamma(p,q)$ is the set of rectifiable curves from $p$ to $q$. Any $\gamma$ realizing this minimum is called a $w$-geodesic.
\end{definition}

For the standard weighting, we can derive a closed form for $w$-geodesics, leading to the following result. The proof requires iteratively reconfiguring polygonal paths, and is in the Appendix.

\begin{proposition}\label{prop:geo_std_len}
Let $w$ be the standard weighting and let $p_0 = (x_0, y_0)$ and $p_1 = (x_1, y_1)$. Assume without loss of generality that $w(p_0) \leq w(p_1)$. Then the $w$-path distance from $p_0$ to $p_1$ is
$$
\min \left( w(p_0)\cdot \frac{|y_1+x_1-y_0-x_0|}{2} + \frac{w(p_1)^2-w(p_0)^2}{2\sqrt{2}},\ \frac{w(p_0)^2 + w(p_1)^2}{2\sqrt{2}} \right)
$$

\end{proposition}
\begin{proof}
    See Appendix \ref{sec:apdx-geo-sl}.
\end{proof}

Proposition~\ref{prop:geo_std_len} gives us a closed form for the minimum line integral for a vine between two points with the standard weighting function. This gives us a new cost function which can be incorporated into a new Wasserstein-type distance.

\begin{definition}
    Let $P$ and $Q$ be persistence diagrams. The \textbf{minimum vine cost (MVC) from $P$ to $Q$}, denote by $\mathrm{MVC}(P, Q)$, is
        $$
        \min_{\pi:A\to B}
    \left(\sum_{\alpha\in A} d_{\text{path}}^{(w)} (\alpha,\pi(\alpha))+
    \sum_{\alpha\notin A} d_{\text{path}}^{(w)}(\alpha,\Delta)+
    \sum_{\beta\notin B} d_{\text{path}}^{(w)}(\beta,\Delta)
    )\right)
$$ 
where $\pi$ ranges over all partial bijections from $P$ to $Q$ and $w$ is the standard weighting.
\end{definition}

Here, by $d_{\text{path}}^{(w)}(\alpha,\Delta)$ we mean the line integral of a path from $\alpha$ to the the closest point (by $d_\infty$ distance) on the diagonal. Since any path can travel along the diagonal with zero cost, this is the same as the minimum line integral for a path from $\alpha$ to any point on the diagonal. It follows that the minimum vine cost is the lowest possible $w$-weighted vineyard distance that any vineyard from $P$ to $Q$ can achieve. We thus have the following theorem.

\begin{theorem}
    Let $\mathcal{V}$ be a vineyard, and let $P = \mathcal{V}(0)$ and $Q = \mathcal{V}(1)$. Let $w$ be the standard weighting. Then, the $w$-weighted vineyard distance $\mathbb{V}^{(w)}(\mathcal{V})$ satisfies
$$
MVC(P, Q) \leq  \mathbb{V}^{(w)}(\mathcal{V}) 
$$
\end{theorem}

\section{Numerical experiments}\label{sec:numeric}

We present two experiments designed to illustrate the difference in behavior between vineyard distance and two comparators. 

\subsection{Mean and variance-shifted Gaussians}\label{sec:gaussians}

In this experiment, $f$ is a 2D Gaussian plotted as a $100\times 100$ image with a standard deviation of $5$ pixels, scaled to have maximum value $1$. The function $g$ is the same but with a different mean or a different standard deviation. For each $g$, we compute the $L^1$ distance, the $1$-Wasserstein distance, and the unweighted straight-line homotopy vineyard distance between $f$ and $g$, using $H_1$ persistent homology for the latter two. Our choice of $H_1$ is designed to detect the ``loop'' around the Gaussian. We divide the $L^1$ distance by the number of pixels to make it numerically comparable to the other distances. 

\begin{figure}
    \centering
    \subfloat{
    \begin{tikzpicture}
        \node at (0.25,0) {\includegraphics[width=0.35\textwidth]{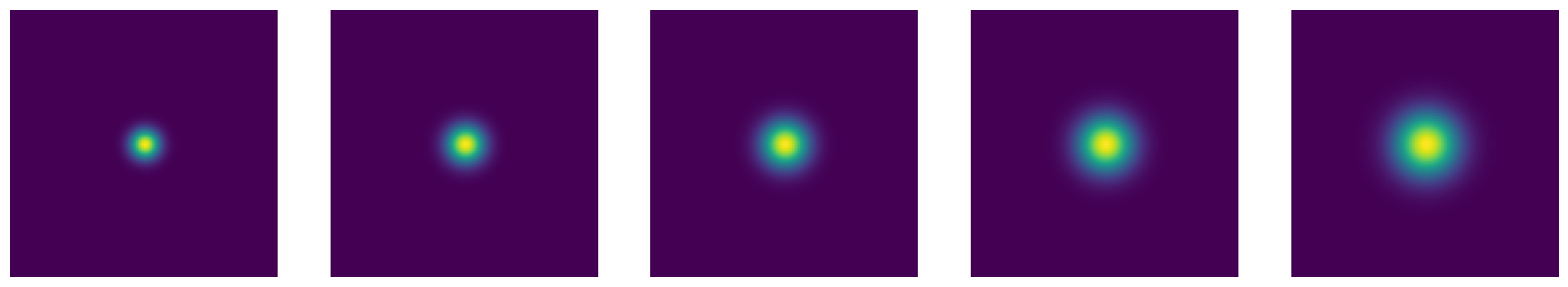}};
        \node at (0,-3) {\includegraphics[width=0.4\textwidth]{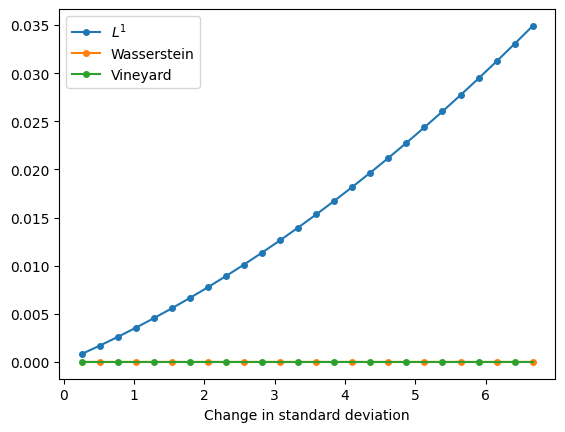}};
    \end{tikzpicture}   
    }%
    \subfloat{
    \begin{tikzpicture}
        \node at (0.25,0) {\includegraphics[width=0.35\textwidth]{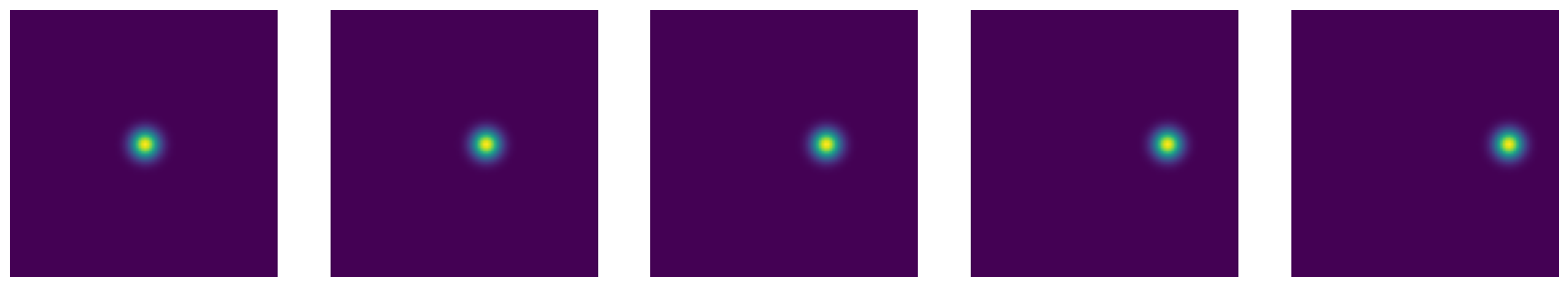}};
        \node at (0,-3) {\includegraphics[width=0.4\textwidth]{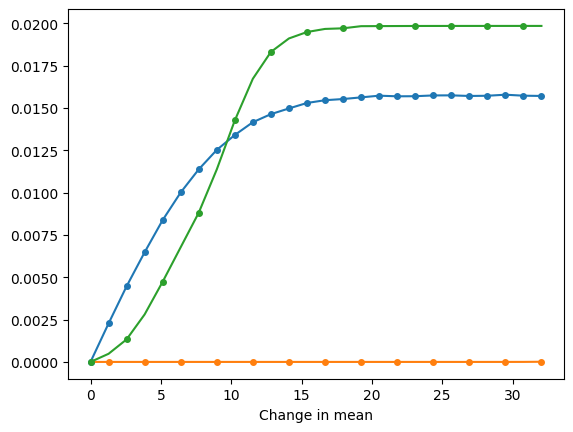}};
    \end{tikzpicture}   
    }%
    \caption{Comparing a Gaussian with a variance shifted version (left) or mean shifted version (right), using scaled $L^1$ distance, the persistence distance, and vineyard distance. Shifted functions are visualized above each plot.}
    \label{fig:gaussians}
\end{figure}

Figure~\ref{fig:gaussians} shows the results. We are mainly interested in the qualitative behavior of each distance. Wasserstein distance is approximately zero under both standard deviation and mean shifts since the birth of the loop (the min value) and the death of the loop (the max value) remain unchanged as long as the bulk of the function is away from the boundary. The $L^1$ distance increases with both kinds of shift as expected, since the individual pixel values change. Note that this particular $L^1$ distance is not the upper bound in Theorem~\ref{thm:vd_upper_bound} since it is divided by the number of pixels. Only vineyard distance shows different behaviors between the two kinds of shift. Shifting the mean changes the location of the loop, leading to the interpolating functions having interesting topology. Shifting the standard deviation, however, creates a family of interpolating functions which all have a loop with about same birth and death time. This simple experiment demonstrates that vineyard distance is qualitatively different from $L^1$ and Wasserstein distances in what kinds of changes it detects. Moreover, vineyard distance differs qualitatively from a fixed linear combination of $L^1$ and Wasserstein distance, since it increases in one case and remains at zero in another.

\subsection{Experiments on image data}\label{sec:app-images}

In this experiment, we use the MNIST dataset consisting of images of digits $0$--$9$. We take 100 images each from the set of $6$s, $7$s and $9$s, treating each as a function $f : \{1, \dots, m\} \times \{1, \dots, n\} \;\to\; [0,256]$, where $f(i,j)$ is the intensity of the pixel at column $i$, row $j$. For each pair of images, we compute the $L^1$ distance, the $1$-Wasserstein distance and the unweighted straight-line homotopy vineyard distance $\mathbb{V}(f,g)$. When computing the latter two distances, we use $H_1$ homology and we invert the image and ``filter down'', so holes appear first and are captured as persistent features. Using multi-dimensional scaling (MDS), we embed the data as points in the plane whose pairwise distances approximate those we computed. Figure~\ref{fig:digits} shows the results.

We see that $1$-Wasserstein distance is unable to distinguish between $6$s and $9$s since they are almost the same up to rotation. On the other hand, $L^1$ distance does not separate the $7$s and $9$s since they have bright pixels in similar areas (at the top of the digit), and the loop in the $9$ is not detected by $L^1$ distance. Vineyard distance separates the three classes much better, drawing a distinction between $6$s and $9$s based on the location of the hole, and $7$s based on the absence of a loop. As of yet, we are not arguing for vineyard distance as an efficient method for image classification since it is more computationally expensive than either of the comparators used here. However, this experiment demonstrates the qualitative differences between the three distances. 

\begin{figure}
    \centering
    \subfloat[$L^1$ distance]{
    \centering
        \includegraphics[height=3.5cm]{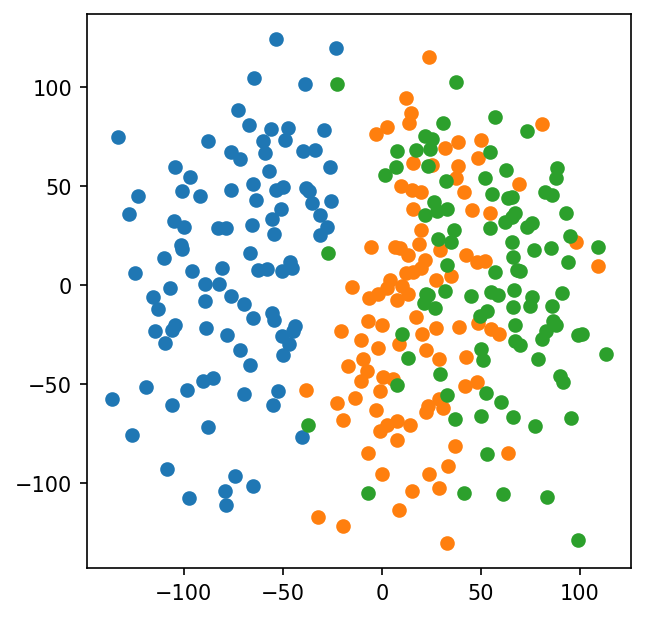}
    }
    \subfloat[$W_1$ distance]{
    \centering
        \includegraphics[height=3.5cm]{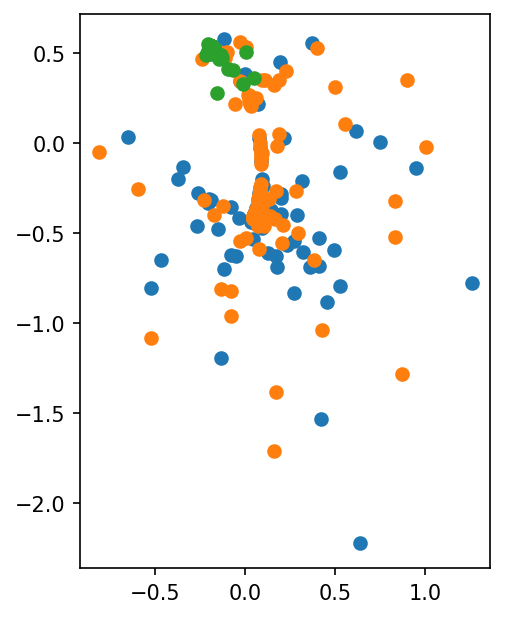}
    }
    \subfloat[Vineyard distance]{
    \centering
        \includegraphics[height=3.5cm]{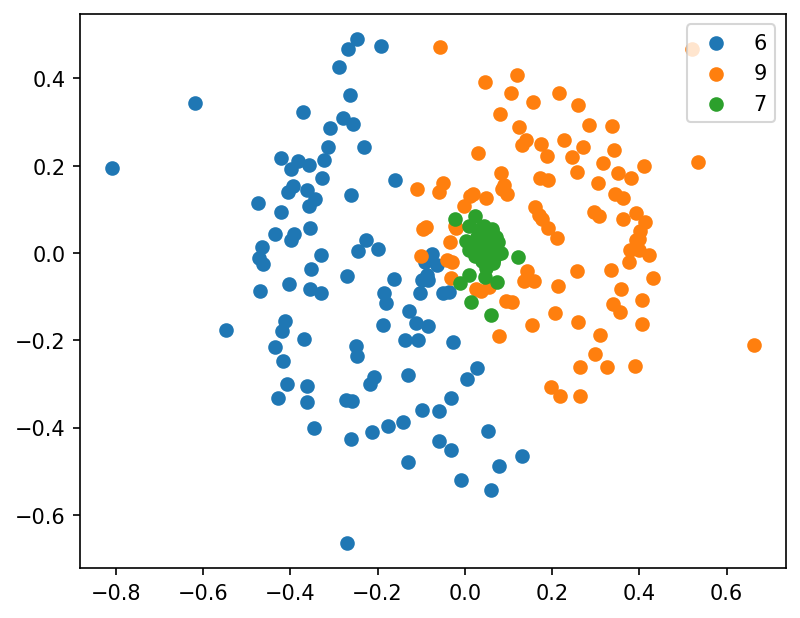}
    }
    \caption{Approximate embeddings using MDS and three different distances on a set of $100$ images of $6$s, $9$s and $7$s. }
    \label{fig:digits}
\end{figure}

\section{Application to geospatial data}\label{sec:app-geospatial}

In this section, we demonstrate how vineyard distance can bring a unique perspective when comparing geospatial datasets. A typical dataset of interest is a set of geographic units and a scalar value associated to each unit (e.g. counties and their Democratic vote share in a recent election). The data is simplified into a \textbf{dual graph} $G = (V_G, E_G)$, wherein each vertex represents a geographic unit, and edges connect vertices if units are adjacent. The scalar values then become a function on the vertices of the dual graph and thus can be used to compute a sublevel set filtration per Section~\ref{sec:vertexbased}. For the remainder of this section, we focus on $0$-dimensional persistent homology.

In what follows, we compute the standard weighting straight line homotopy vineyard distance $\mathbb{V}^{(w)}(\mathcal{V}(f,g))$ for various pairs of functions. Just as in our numerical experiments, we compute two comparators:
\begin{itemize}
    \item the \(L^1\) distance $||f-g||_{L^1} = \frac{1}{|V_G|}\sum_{v\in V_G}|f(v) - g(v))|$, and
    \item the $1$-Wasserstein distance on persistence diagrams $W_1(f,g)$ = $W_1(\dgm_0(f),\dgm_0(g))$.
\end{itemize}

We use real-world examples to tease out the similarities and differences between these three distances. Our overall conclusion is that for anti-correlated or uncorrelated datasets, such as the Black and Hispanic populations in Section~\ref{subsec:cities}, vineyard distance can reveal distinctions hidden by both $L^p$ and persistence distance. When datasets are highly correlated, such as Black population and Democratic voters in Section~\ref{subsec:counties}, vineyard distance acts as an effective combination of $L^p$ and persistence distance and is correlated with both.

\subsection{City demographics}\label{subsec:cities}

We begin with some examples from the U.S. cities dataset analyzed in~\cite{kauba2024topological}. This dataset consists of the $100$ largest U.S. cities with boundaries from the CDC's $500$ city project~\cite{CDC500CitiesBoundaries}, and demographic data obtained from the $2020$ Census~\cite{IPUMS2021}. For each city, we define functions $f$ and $g$ on the dual graph as follows: 
\begin{itemize}
    \item $f(v) = 1-R_B(v)$, where $R_B(v)$ is the fraction of tract $v$ who identified as Black.
    \item $g(v) = 1-R_H(v)$, where $R_H(v)$ is the fraction of tract $v$ who identified as Hispanic.
\end{itemize}%
Note that the U.S.~Census treats Black as a race category, and Hispanic as an ethnicity, so that these categories are not mutually exclusive. Nonetheless, we expect these distributions to be very disjoint in practice. Also note that we are ``filtering down'' so the persistence diagram will detect local maxima 

We first examine a case where the vineyard distance detects a distinction hidden by $1$-Wasserstein distance. As shown in Figure~\ref{fig:cities-Milwaukee_Lexington_4x4} and Table~\ref{tab:cities}, Milwaukee, WI and Lexington, KY have similar $W_1(f,g)$ but very different $||f-g||_{L^1}$. This is because the Black and Hispanic distributions, while similar in shape, are well-separated geographically in Milwaukee but not in Lexington. Table~\ref{tab:cities} shows that vineyard distance successfully detects the difference in this case. 

We also present a case where neither $1$-Wasserstein distance nor $L^1$ distance detects a distinction, but vineyard distance does. Returning to Table~\ref{tab:cities}, Phoenix, AZ and Dallas, TX have similar $W_1(f,g)$ and $||f-g||_{L^1}$, but their vineyard distances differ by a factor of almost $2$. We observe the reason for this in Figure~\ref{fig:cities-Phoenix_Dallas_4x4}. The shape and individual tract values of $f$ and $g$ are very different in both cities, creating a high $1$-Wasserstein and $L^1$ distance. However, the geographic overlap between $f$ and $g$ in Phoneix yields a smooth linear homotopy, resulting in a lower vineyard distance compared to Dallas. In contrast to the classical alternatives, only vineyard distance captures this persistent topological similarity.

\begin{table}
    \centering
    \begin{tabular}{c|c|c||c|c|}
        & Milwaukee, WI &  Lexington, KY & Phoenix, AZ & Dallas, TX\\
        \hline 
       $W_1(f,g)$  & 0.564 & 0.593 & 3.724 & 3.844\\
        \hline 
       $||f-g||_{L^1}$ & 0.389 & 0.075 & 0.315 & 0.271\\
        \hline 
       $\mathbb{V}^{(w)}(\mathcal{V}(f,g))$ & 0.006 & 0.001 & 0.011 & 0.021\\
        \hline 
       \end{tabular}
    \caption{Comparison of distances between Black population percentage ($1-f$) and Hispanic population percentage ($1-g$) data for two pairs of cities.}
    \label{tab:cities}
\end{table}

\begin{figure}
    \centering
    \includegraphics[width= .65 \textwidth]{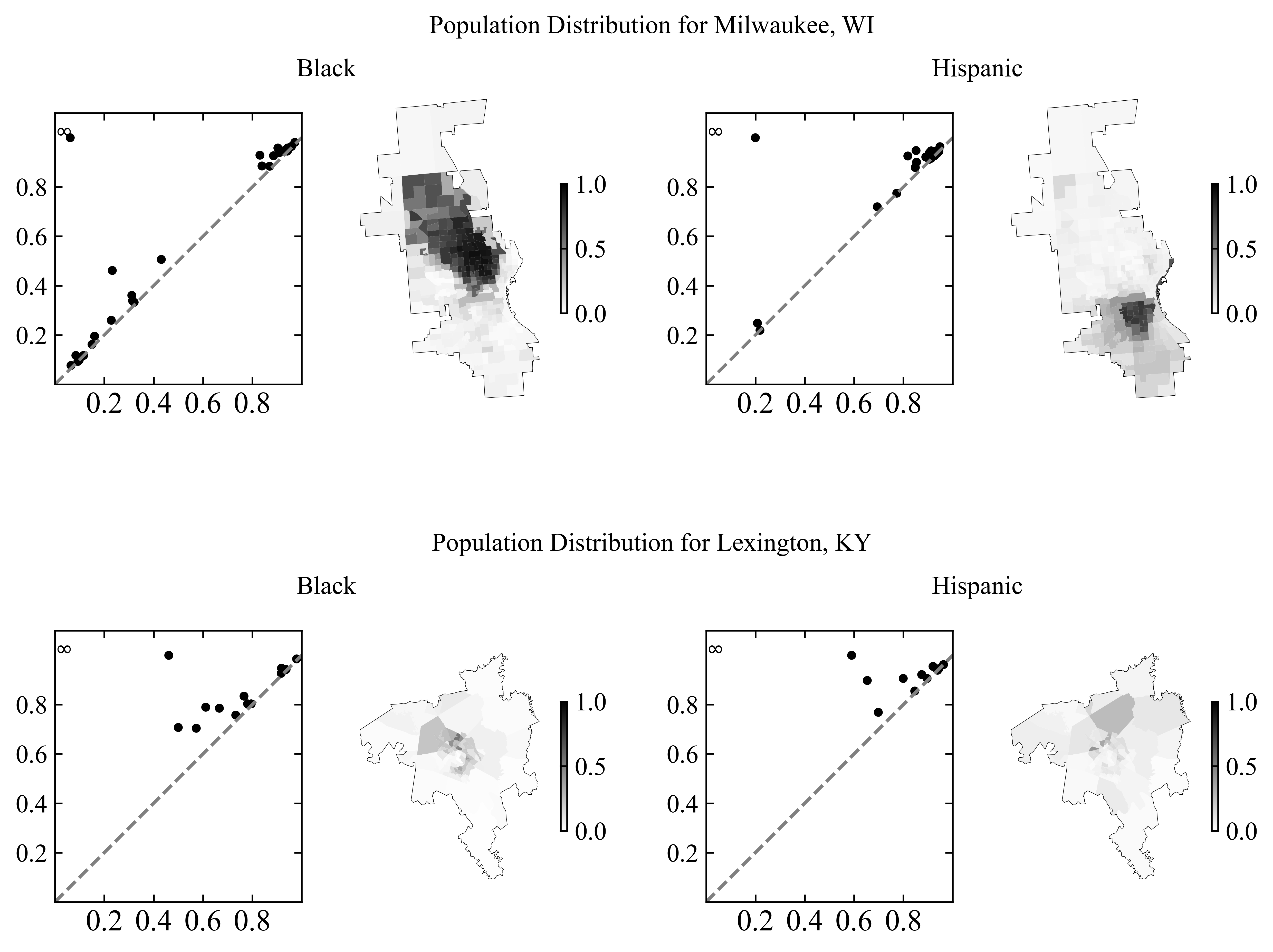}
    \caption{We show the persistence diagrams for $f$ and $g$ as defined in Section~\ref{subsec:cities} beside the data that these functions depend on (Black and Hispanic population percentages), shown on the map of each city. We show two cities: Milwaukee (above) and Lexington (below). }
    \label{fig:cities-Milwaukee_Lexington_4x4}
\end{figure}

\begin{figure}
    \centering
    \includegraphics[width=0.65\textwidth]{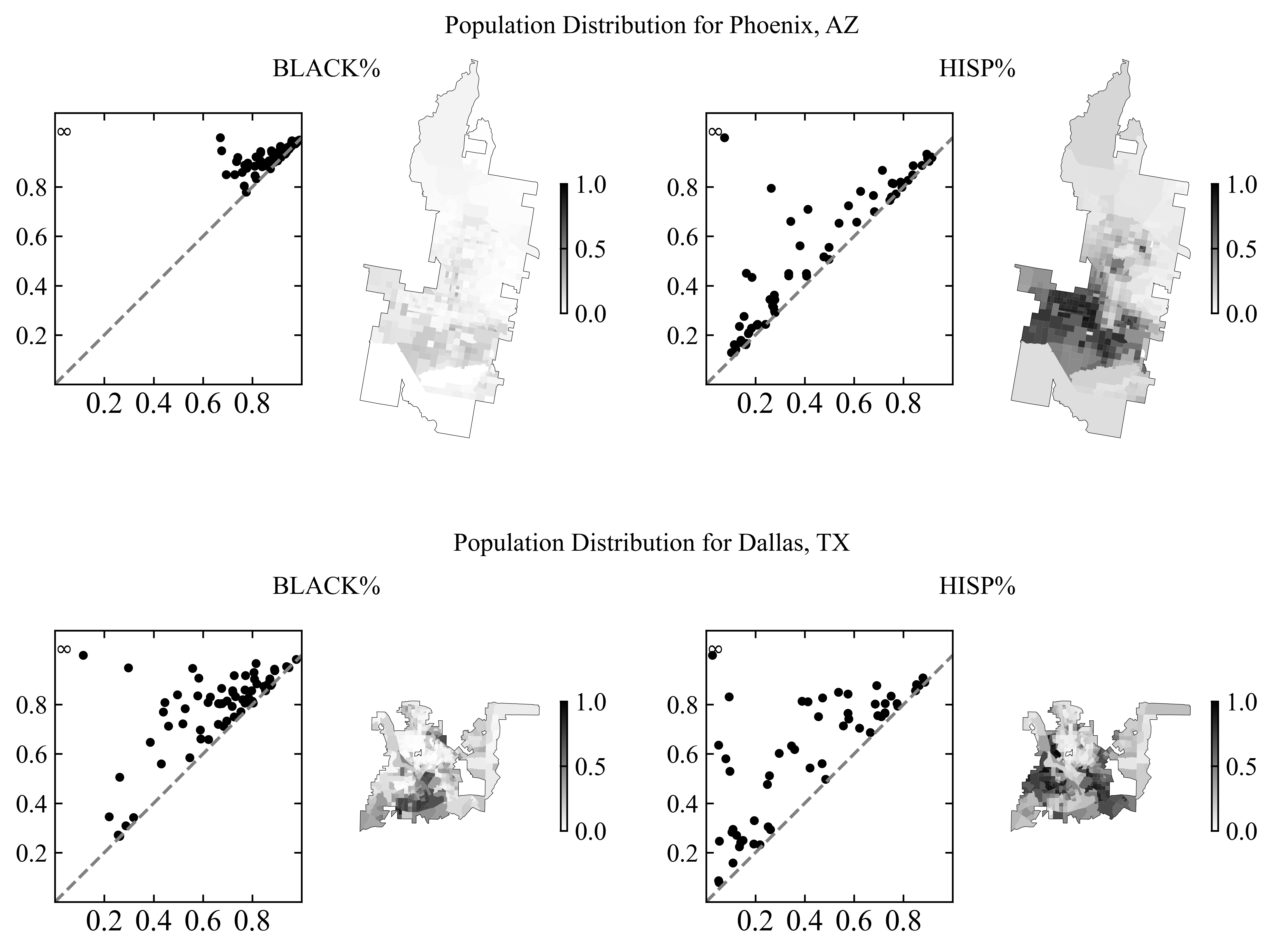}
    \caption{We show the persistence diagrams for $f$ and $g$ as defined in Section~\ref{subsec:cities} beside the data that these functions depend on (Black and Hispanic population percentages), shown on the map of each city. We show two cities here: Phoenix (above) and Dallas (below). }
    \label{fig:cities-Phoenix_Dallas_4x4}
\end{figure}

\subsection{Voting data}\label{subsec:counties}

In this section, we perform a complementary analysis using U.S.~election data instead of just demorgraphic data. We use county vote totals for six Presidential elections ($2000$--$2020$) from the MIT Election Data + Science Lab~\cite{HarvardElections2018}. For each dual graph, we define our functions $f$ and $g$ as follows:

\begin{itemize}
    \item $f(v) = 1 - R_B(v)$, where $R_B(v)$ is the fraction of county $v$ who identify as Black.
    \item $g(v) = 1 - R_D(v)$, where $R_D(v)$ is the fraction of county $v$ who voted Democrat.
\end{itemize}

This choice of functions is motivated by the consistent (but nuanced and well-studied) support of Black voters for the Democratic party in recent decades~\cite{haynie2010blacks}. We therefore expect to see strong correlation between these functions, modulated by whether there are other kinds of Democratic voters available, creating an interesting case study for measuring similarity between two functions.

In Figure~\ref{fig:election-scatter}, we observe a strong linear correlation between $\mathbb{V}^{(w)}(\mathcal{V}(f,g))$ and $W_1(f,g)$. This demonstrates that vineyard distance encodes similar topological information to $1$-Wasserstein distance in the high correlation setting. By shading each point in Figure~\ref{fig:election-scatter} by its $L^1$ distance, we see that within a fixed $1$-Wasserstein band, vineyard distance increases with $L^1$ distance, demonstrating that vineyard distance acts like a combination of $L^1$ and $1$-Wasserstein distance as expected.

It is instructive to look at the greatest vineyard distance in our dataset, namely Virginia's $2020$ election results shown in Figure~\ref{fig:election-VA_2020}. The areas of Democratic support outside of black voter hubs lead to different shapes and geographic locations for the $f$ and $g$, which create a high vineyard distance. Looking at the right hand plot of Figure~\ref{fig:election-scatter}, we see that broad Democratic support tends elevates vineyard distance, precisely because non-Black Democratic voter groups also affect the analysis.

\begin{figure}
\centering
    \centering
    \includegraphics[width=\textwidth]{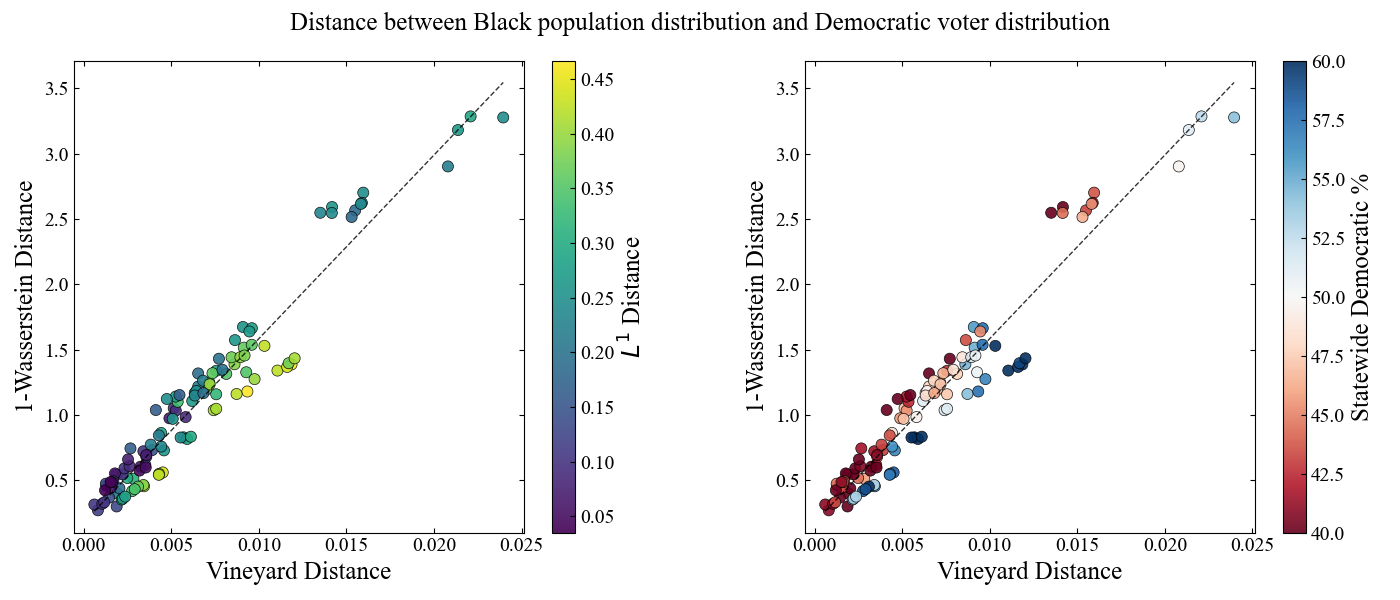} 
\caption{Comparing the Black population distribution with the distribution of Democratic voters across 50 U.S. States and six elections. We plot the Vineyard Distance against the $1$-Wasserstein distance in both plots. On the left, we color the points by a third distance, the $L^1$ distance, and on the right we color the points by statewide share, each of which seems to affect some of the deviation from the line $y=x$.}\label{fig:election-scatter}
\end{figure}

\begin{figure}
    \centering
    \includegraphics[width=0.6\textwidth]{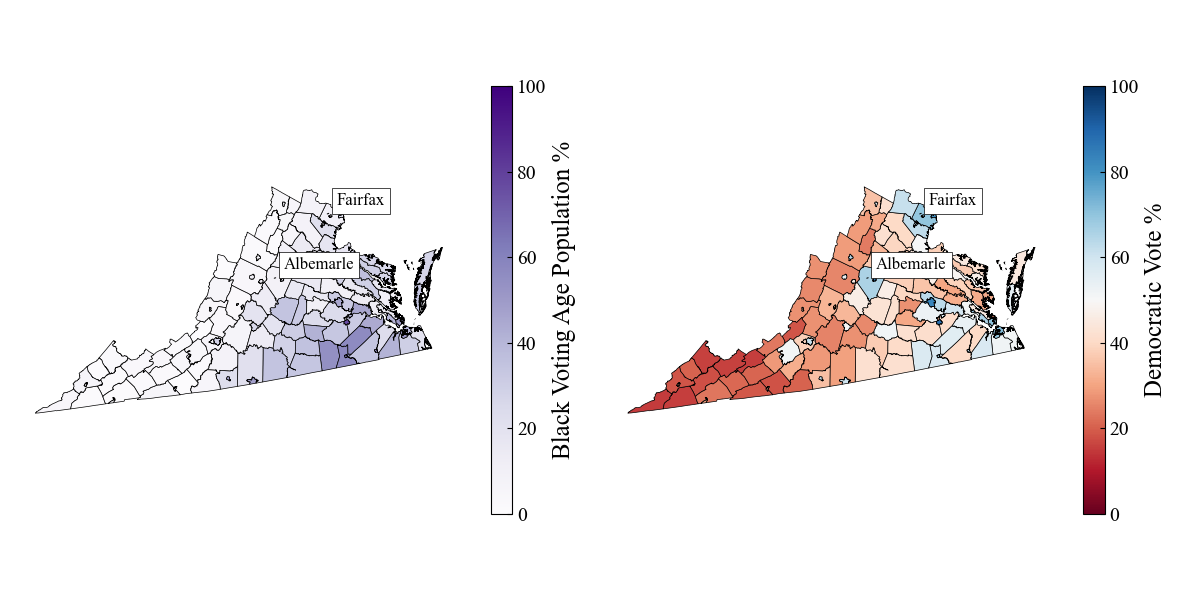}
    \caption{County-level map of Virginia. Counties are shaded according to their \% black voting age population (left) and \% democrat vote (right). Counties with high Democratic vote share outside of black voting age population hubs are annotated.}
    \label{fig:election-VA_2020}
\end{figure}

\section{Application to deep learning}\label{sec:app-nn}

\subsection{Motivation}

We now present an experiment applying vineyard distance to the study of neural network training dynamics. A fundamental task in deep learning is to fit a neural network to training data to build a classifier. Through gradient descent, internal parameters in the neural network are adjusted to reduce the training loss until a convergence criterion is satisfied. Despite their enormous practical success, the internal dynamics of neural network training remain difficult to interpret and compare across different architectures.

Each gradient descent step is called an epoch. If we consider the space of all possible inputs $X$, and focus on binary classification, then at each epoch the neural network gives a \textbf{output function} $f: X \to [0,1]$. A perfect classification would send all points in Class 0 to $0$ and all points in Class $1$ to $1$. In practice, however, the image of $f$ contains more than just $\{0,1\}$; we can think of $f(x)$ as a prediction representing the confidence that the point $x$ is in Class 1 and not Class 0. We will use vineyard distance to study the evolution of the output function throughout the training process.

\subsection{Experiments}

\begin{figure}[h]
    \centering
    \subfloat[Histogram of vineyard distances from training trajectories pooled across all setups and all trials.]{
    \centering
        \includegraphics[height=4cm]{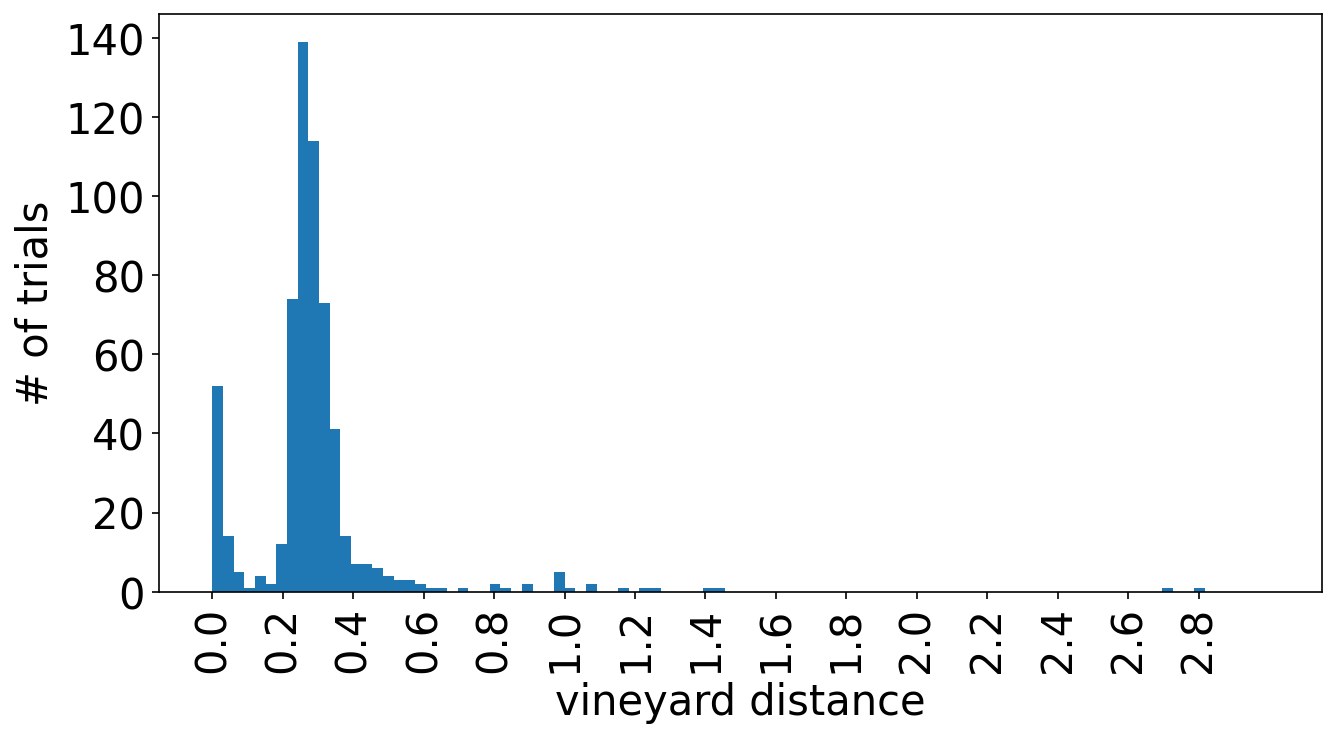}\label{fig:alltogetherhistogram}
    }  
    \subfloat[Vineyard distances for successful training trajectories, broken down by the hyperparameter value that was varied.]{
    \centering
        \includegraphics[height=4cm]{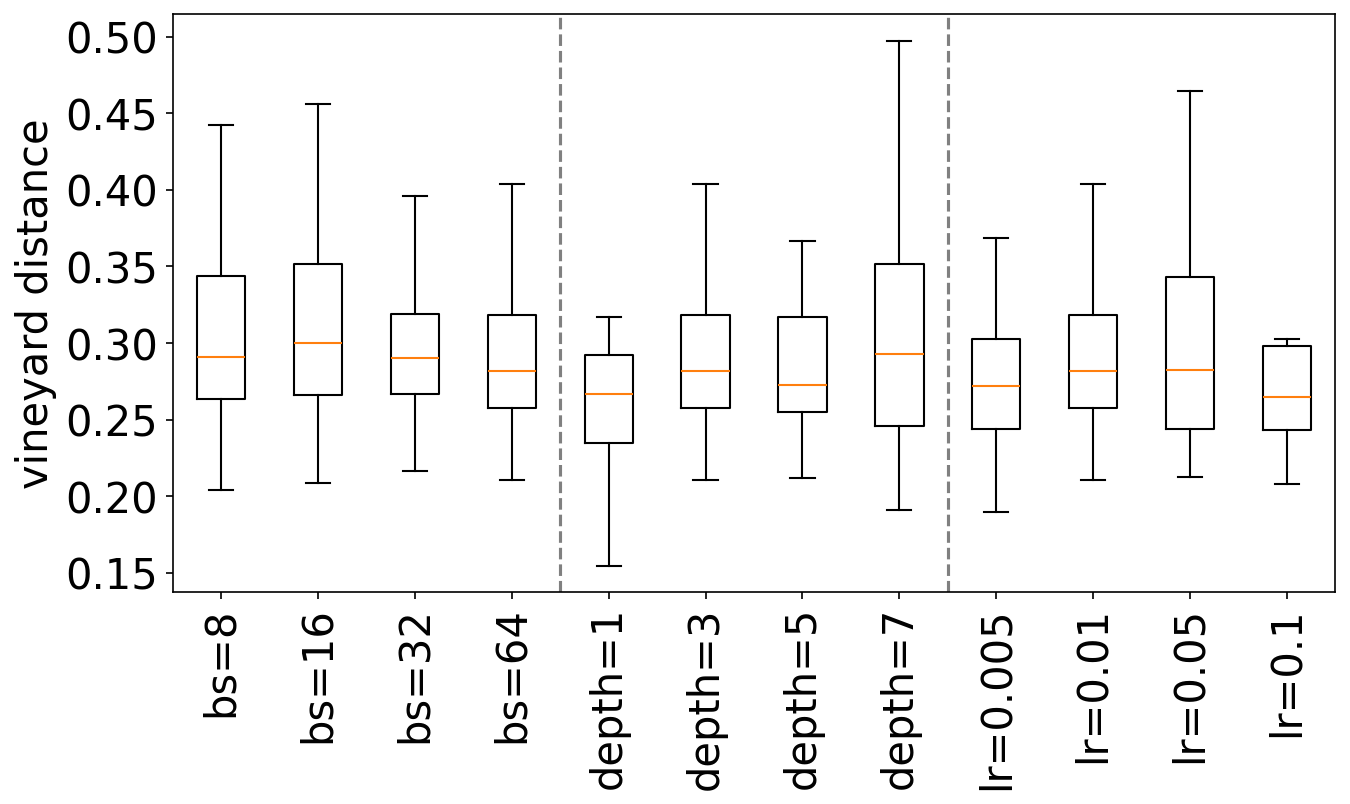}\label{fig:alltogetherboxplot}
    }
    \caption{Summary of vineyard distances given by the training trajectory of various neural network setups.}
    \label{fig:alltogether}
\end{figure}

For our experiments, we use the toy dataset in two dimensions shown in Figure~\ref{fig:exampletraj}. In practical applications, the dimension and size of input data is typically much larger, but we will use our dataset as a proof of concept that has the advantage of being easy to visualize and which has an easily identifiable topology for the decision boundary (a circle). At each epoch, we represent the output function as an image (see Figure~\ref{fig:exampletraj}) and compute its $H_1$ persistence. To obtain a standard weighting vineyard distance, we sum the weighted Wasserstein distance between successive persistence diagrams (i.e.~a finite $n$ version of Theorem~\ref{thm:distancebyW}). Note that this is different to the straight-line homotopy case since the (discrete) vineyard is defined by the training process not an interpolation.

While most analyses of neural network performance focus on the speed with which the loss decreases and the final accuracy, we are interested in qualitative features of the training process itself. To cover a range of possible scenarios, we vary three hyperparameters: the batch size (\textbf{bs}), the number of hidden layers (\textbf{depth}), and the learning rate (\textbf{lr}). Our base architecture is $\text{bs} = 64$, $\text{depth} = 3$ and $\text{lr} = 0.01$, and we vary only one of these hyperparameters at a time. For each setup, we run 100 trials, choosing a random 80\% of the points in the dataset each time,  and record the vineyard distance.

\begin{figure}[h]
    \centering
    \subfloat[Dataset]{
    \centering
        \includegraphics[width=0.12\textwidth]{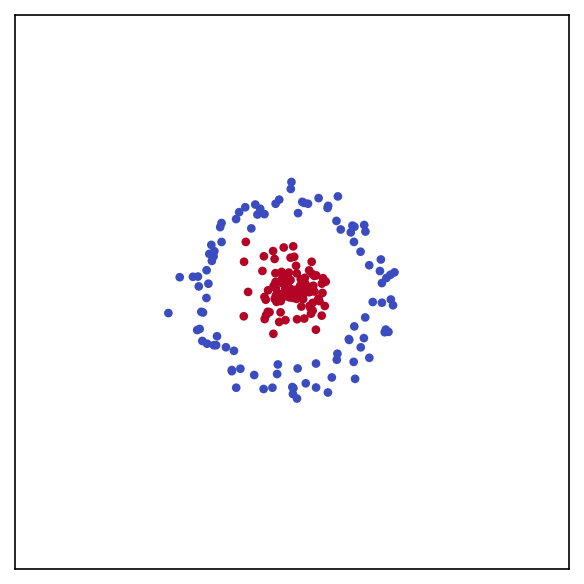}
    }
    
    \subfloat[Sample training trajectories of increasing vineyard distance.]{
    \centering
    \begin{tikzpicture}[scale=0.95]
        \node at (3,5) {\includegraphics[width=9.5cm]{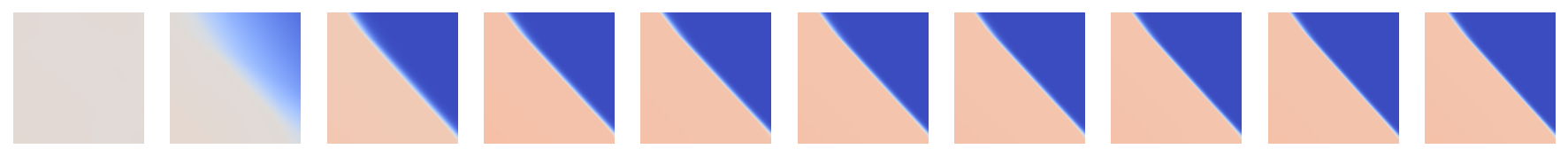}};
        \node at (3,3) {\includegraphics[width=9.5cm]{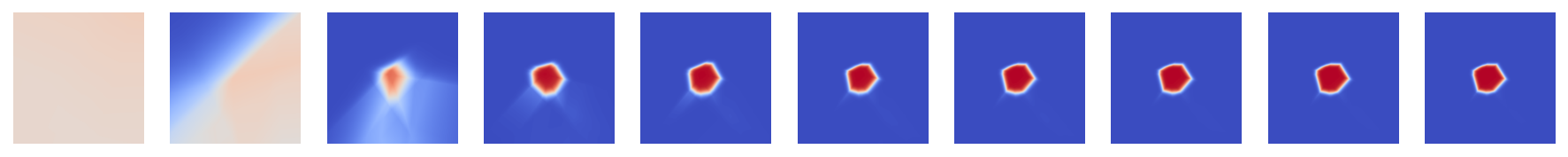}};
        \node at (3,1) {\includegraphics[width=9.5cm]{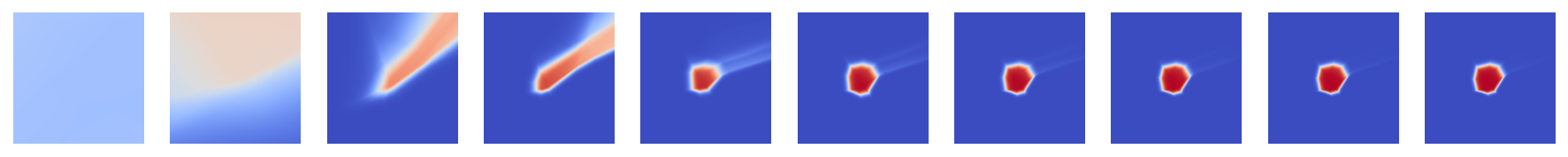}};
        \node at (3,-1) {\includegraphics[width=9.5cm]{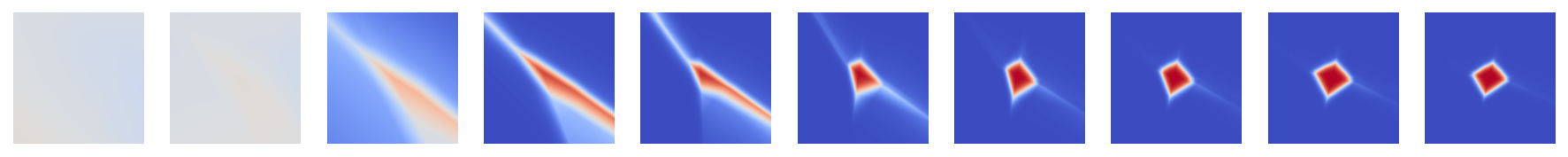}};
        \node at (3,-3) {\includegraphics[width=9.5cm]{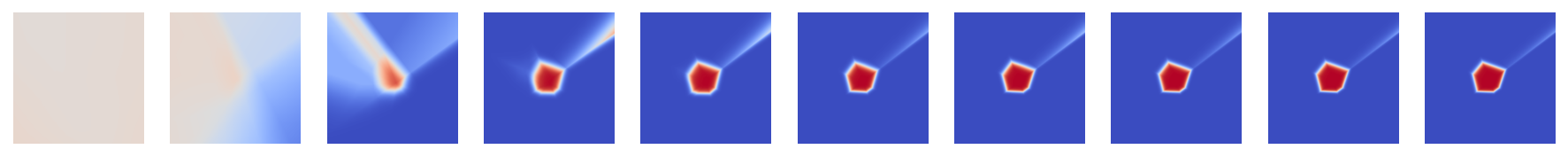}};
        \node at (3,-5) {\includegraphics[width=9.5cm]{nn/example_traj4.png}};
        \node at (3.5,-6) {\small training};
        \draw[->] (-1.5,-5.7)--(7.5,-5.7);
        \node at (9.2,5) {\includegraphics[width=1.8cm]{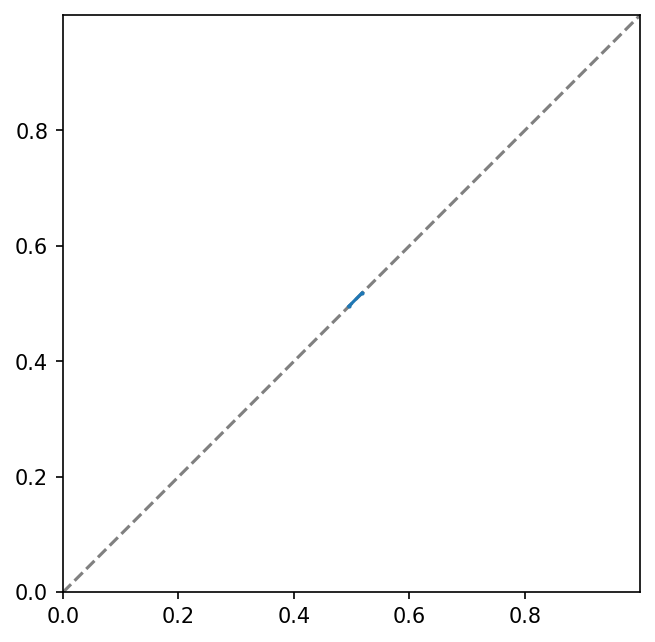}};
        \node at (9.2,3) {\includegraphics[width=1.8cm]{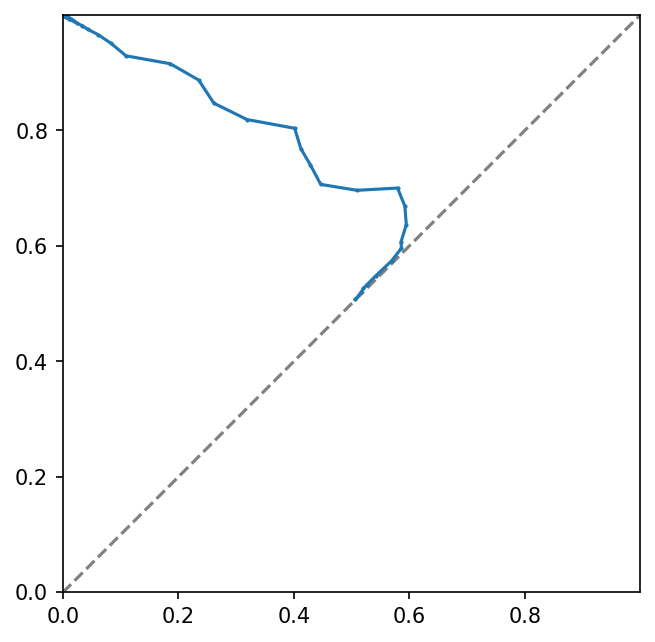}};
        \node at (9.2,1) {\includegraphics[width=1.8cm]{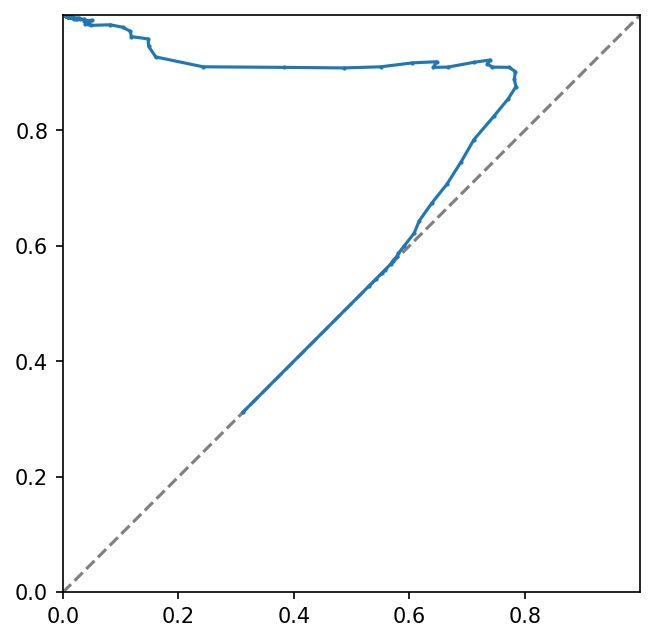}};
        \node at (9.2,-1) {\includegraphics[width=1.8cm]{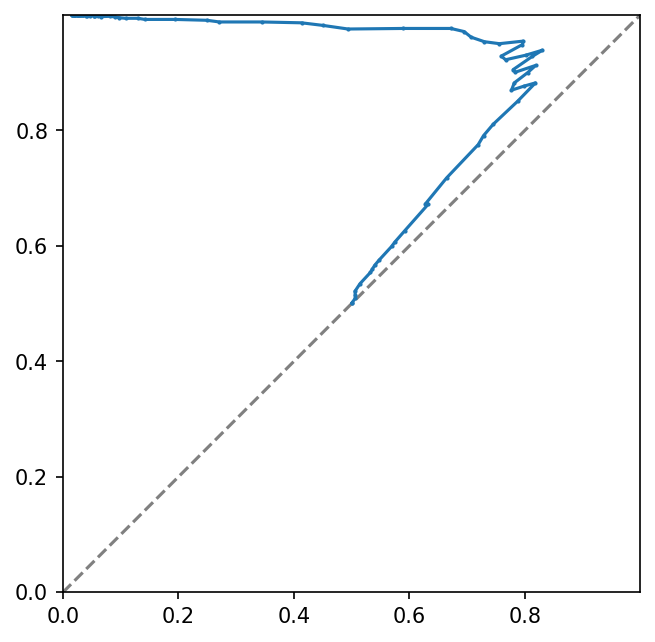}};
        \node at (9.2,-3) {\includegraphics[width=1.8cm]{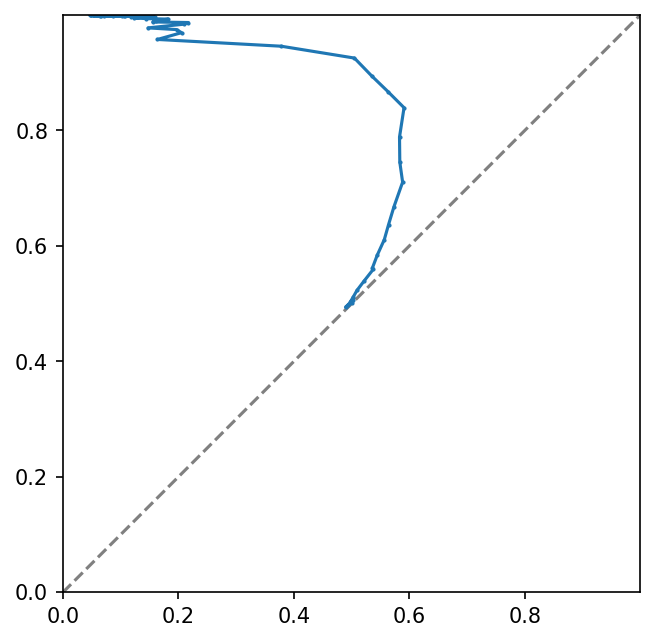}};
        \node at (9.2,-5) {\includegraphics[width=1.8cm]{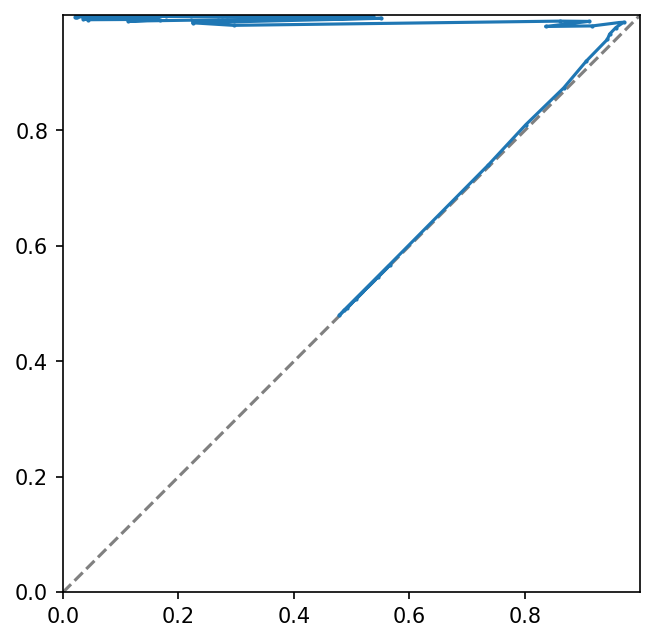}};
        \node at (11,5) {$\mathbb{V} = 0.000$};
        \node at (11,3) {$\mathbb{V} = 0.250$};
        \node at (11,1) {$\mathbb{V} = 0.351$};
        \node at (11,-1) {$\mathbb{V} = 0.439$};
        \node at (11,-3) {$\mathbb{V} = 0.970$};
        \node at (11,-5) {$\mathbb{V} = 2.716$};
    \end{tikzpicture}
    }
    \caption{Some example training trajectories for the our toy dataset. We show ten snapshots of the output function during training for each trial. On the right we plot the vineyard and give the vineyard distance. The vineyard distance distinguishes not only between successful and unsuccessful training, but also between different evolutions in the shape of the output function.}
    \label{fig:exampletraj}
\end{figure}

A histogram of all vineyard distances is shown in Figure~\ref{fig:alltogether}. We see a clear division of values into three main categories: values near $0$, a range of values between $0.2$ and $0.4$ and some outlier values above $0.6$. In Figure~\ref{fig:exampletraj}, we display more information about a sample learning trajectory from these bands. We see that values near $0$ indicate that the neural network was unable to find the correct decision boundary at all. Values between $0.2$ and $0.4$ indicate successful training. Lower values in this band indicate that the shape of the boundary was fixed early and then sharpened, while higher values indicate that the boundary was sharpened before the red region was fully enclosed. These two behaviors produce distinct vines in $H_1$, as Figure~\ref{fig:exampletraj} shows. Finally, values higher than $0.5$ correspond to oscillations in the topology of the decision boundary, sometimes at the very end of the training cycle, where training accuracy has already reached 100\%.

In Figure~\ref{fig:alltogetherboxplot} we take only those trials that ended with accuracy above 95\% and show a boxplot breaking down the vineyard distances by setup, with outliers removed. The result is somewhat surprising, namely that none of the hyperparameters has a large overall effect on the vineyard distance of a typical successful training trajectory. In other words, the randomness that dominates is instead from the subsampling of the data and from the random initialization of parameters. We note that the independence of vineyard distance from the parameterization of the vines is what makes this an apples-to-apples comparison. Further investigation is needed to determine exactly what causes the difference between training trajectories in persistence space.

\section{Conclusion}

Because vineyards are complex objects, they have historically been used for exploratory analysis wherein a vineyard is computed, plotted, and then analyzed carefully to identify trends. In order to study large ensembles of vineyards, such as for $100$ cities as in Section~\ref{sec:app-geospatial} or many different training trajectories as in Section~\ref{sec:app-nn}, vineyards need to be characterized by scalar properties. This is somewhat analogous to the use of persistence statistics to summarize ensembles of persistence diagrams (see e.g.~\cite{chung2021persistent}). This paper proposes vineyard distance -- the weighted length of the image of the vineyard in the plane -- as a natural choice for summarizing a vineyard with a scalar. Our other main proposal is that in the absence of an existing vineyard interpolating between two functions, a straight-line homotopy can be used. This way of measuring distance between functions has distinct advantages since it resembles a hybrid distance: lying somewhere between $L^p$ and persistence distances.

The main drawback of using vineyard distance is that it is even more computationally expensive than computing persistence diagrams. This is somewhat mitigated by the fact that vineyards can be computed in linear time rather than recomputing the persistence diagram at each infinitesimal time step for smooth homotopies, as noted in the original vineyard paper~\cite{cohen2006vines}. In the straight-line homotopy case, the critical points (where the filtration order of simplices changes) are predictable in advance, which suggests an efficient algorithm is possible. We leave this to future work, along with other topological distance measures inspired by vineyards.

\FloatBarrier

\newpage

\section*{Data availability statement}

Datasets for the geospatial data applications are in the public domain: the CDC city boundaries~\cite{CDC500CitiesBoundaries}, Census 2020 demographics obtained from NHGIS~\cite{IPUMS2021}, and election data~\cite{HarvardElections2018}. Python code for reproducing the experiments in this paper is available at \url{https://github.com/arulandu/vineyard/}.

\section*{Acknowledgments} We would like to thank Sat Gupta, Jianping Sun, Leslie Justice, and Denise Greenwell for their efforts in organizing the Research Experience for Undergraduates where this research took place.

\section*{Funding}
This work is supported in part by funds from the National Science Foundation grant DMS-2244160.

\appendix 

% \section{Computational details}\label{sec:apdx-comp}

% Persistence diagrams were computed using the \texttt{gudhi} library. All neural network experiments were conducted using \texttt{pytorch}. 

\section{Proof of Proposition~\ref{prop:geo_std_len}} \label{sec:apdx-geo-sl}

We begin with some lemmas. Let $w$ be the standard weighting $w(x,y) = \frac{1}{\sqrt{2}}(y-x)$ throughout this section.

\begin{lemma}
Let $p_0 = (x_0, y_0)$ and $p_1 = (x_1, y_1)$, and let $\gamma: [0,1] \to \mathbb{R}^2$ be the line segment
$$
\gamma(t) = (1-t)p_0 + tp_1
$$
Then,
$$
\int_{\gamma} w ds = d_\infty(p_0, p_1)\frac{w(p_0) + w(p_1)}{2}
$$
\end{lemma}
\begin{proof}
    The integral in question is the average of $w$ over the length of $\gamma$ times the length of $\gamma$, which is equal to the right hand side since $w$ is linear.
\end{proof}

\begin{lemma}\label{lemma:breakup}
    Let $p_0 = (x_0, y_0)$ and $p_1 = (x_1, y_1)$. Suppose without loss of generality that $w(p_0) \leq w(p_1)$. Suppose further that $x_0 + y_0 \leq x_1 + y_1$. Then
    $$
    d_\infty(p_0, p_1) \frac{w(p_0) + w(p_1)}{2} \geq d_\infty(p_0, \tilde{p}) \frac{w(p_0) + w(\tilde{p})}{2} + d_\infty(\tilde{p}, p_1) \frac{w(\tilde{p}) + w(p_1)}{2}
    $$
    where $\tilde{p} = (\frac{x_0 + x_1}{2} + \frac{y_1 - y_0}{2}, \frac{x_1 - x_0}{2} + \frac{y_0 + y_1}{2})$.
\end{lemma}
\begin{proof}
Notice that all the $d_\infty$ terms have closed forms. Indeed, since $p_1 - \tilde{p}$ and $\tilde{p}-p_0$ are parallel to the lines $y=-x$ and $y=x$ respectively, we have
    \begin{align*}
        d_\infty(\tilde{p}, p_1) & = \frac{x_0-x_1}{2} + \frac{y_1 - y_0}{2} \\
        d_\infty(p_0, \tilde{p}) & = \frac{x_1-x_0}{2} + \frac{y_1 - y_0}{2}
    \end{align*}
From the assumptions on $p_0$ and $p_1$, we have that $y_1 - y_0 \geq x_1 - x_0,\ y_1 - y_0 \geq x_0 - x_1$, and $y_1 \geq y_0$, so we have $d_\infty(p_0, p_1) = y_1 - y_0$. Once these have been substituted into the equation above, we have that the difference between the left and the right hand side is 
$$
\frac{(y_0 - y_1)^2 - (x_0 - x_1)^2}{4 \sqrt{2}} \geq 0
$$
\end{proof}

The first assumption in Lemma~\ref{lemma:breakup}, that $w(p_0) \leq w(p_1)$, can be made without loss of generality. It may seem that the second, $x_0+y_0 \leq x_1+y_1$, cannot. However, reflecting $p_0$ across the perpendicular line from $p_1$ to the diagonal, we can always reduce to the case where this second assumption also holds. Thus, Lemma~\ref{lemma:breakup} applies generally to points in $\mathbb{R}^2_{x \leq y}$. 

What Lemma~\ref{lemma:breakup} shows is that any path which is a straight line segment can be decomposed into a parallel and perpendicular (to the diagonal) part, and that the concatenation of these two parts has lower line integral than the original path; see Figure~\ref{fig:decompose}.

\begin{proof}[Proof of Proposition~\ref{prop:geo_std_len}]
Our goal is to show that every path from $p_0$ to $p_1$ in $\mathbb{R}^2_{x\leq y}$ can be modified into one of two canonical forms without increasing the line integral. We first consider a path $\gamma$ from $p_0$ to $p_1$ which intersects the diagonal. Since $w$ is zero on the diagonal, we can modify $\gamma$ to consist of perpendicular line to the diagonal from $p_0$, a line segment along the diagonal, and a perpendicular line from the diagonal to $p_1$. A quick calculation yields
\begin{equation}\label{eq:diagonalcase}
  \int_\gamma w ds = \frac{1}{2}d_\infty(p_0, \Delta)w(p_0) + \frac{1}{2}d_\infty(p_1, \Delta)w(p_1) = \frac{w(p_0)^2 + w(p_1)^2}{2\sqrt{2}}  
\end{equation}
for this case.

Let's assume now that we have a piecewise linear path from $p_0$ to $p_1$ that does not intersect the diagonal. By decomposing every piece as in Figure~\ref{fig:decompose}, we may assume that all the pieces are either parallel to or perpendicular to the diagonal. In particular, each segment is in one of four possible directions, which for convenience we refer to be northwest (NW), northeast (NE), southwest (SW), and southeast (SE). We will start rearranging the order of these segments while never increasing the line integral. Moving SE segments earlier in the order can not increase the line integral since every subsequent segment shifts closer to the diagonal. Similarly, moving NW segments later in the order can not increase the line integral. Thus, we may assume that our path consists of a number of SE segments, followed by NE and SW segments, and finally NW segments. Since backtracking can only increase the integral, we may replace the NE and SW segments by a single line segment in one of those two directions. Thus our path can be modified into the form of Figure~\ref{fig:classicform}. Note that if during these modifications, the path intersects the diagonal, then we are in the case we dealt with earlier in the proof.

Let us assume we are in the left hand case in Figure~\ref{fig:classicform} in which $x_1+y_1 \geq x_0 + y_0$, and let $q$ be the distance between the start and end points of the SE segments. Using the distances derived in the proof of Lemma~\ref{lemma:breakup}, we can write the line integral as
$$
\frac{w(p_0)+w(p_0)-q}{2}\cdot q + \left(w(p_0) - q\right)\cdot \frac{y_1 - y_0 + x_1 - x_0}{2} + \frac{w(p_0) - q + w(p_1)}{2} \cdot \left(\frac{y_1-y_0-x_1+x_0}{2} + q \right)
$$
which is a convex function of $q$. This means that its minima are at either $q = 0$ or at the value of $q$ for which the path intersects the diagonal, which we have already dealt with. For $q = 0$, we have
\begin{align*}
    \int_\gamma w ds & = w(p_0)\cdot \frac{y_1-y_0+x_1-x_0}{2} + \frac{w(p_0)+w(p_1)}{2} \cdot \frac{y_1 - y_0 - x_1 + x_0}{2} \\
    & = w(p_0)\cdot \frac{y_1-y_0+x_1-x_0}{2} + \frac{w(p_1)^2-w(p_0)^2}{2\sqrt{2}}
\end{align*}
If are instead in the right hand case of Figure~\ref{fig:classicform}, then we have
\begin{align*}
    \int_\gamma w ds & = w(p_0)\cdot \frac{y_0-y_1+x_0-x_1}{2} + \frac{w(p_0)+w(p_1)}{2} \cdot \frac{y_1 - y_0 - x_1 + x_0}{2} \\
    & = w(p_0)\cdot \frac{y_0-y_1+x_0-x_1}{2} + \frac{w(p_1)^2-w(p_0)^2}{2\sqrt{2}}
\end{align*}
Since any rectifiable path from $p_0$ to $p_1$ can be approximated by piecewise linear paths, these are the required lower bounds.
\end{proof}

\begin{figure}[h]
\centering
\subfloat[Two examples of a straight path $\gamma$ into a parallel and a perpendicular segment. Lemma~\ref{lemma:breakup} shows that this cannot increase the line integral.]{
    \centering
    \begin{tikzpicture}[>=stealth, thick, scale=0.35]
    \begin{scope}
        \node[above left] at (1,2) {$p_0$};
        \filldraw (1,2) circle[radius=1.5pt];
        \node[above left] at (4,9) {$p_1$};
        \filldraw (4,9) circle[radius=1.5pt];
        \node[right] at (6, 7) {$\tilde{p}$};
        \filldraw (6,7) circle[radius=1.5pt];
        \draw[->] (1,2)--(4,9);
        \node at (2.4, 6) {$\gamma$};
        \draw[dotted, ->] (1,2)--(6, 7);
        \draw[dotted, ->] (6, 7)--(4,9);
    \end{scope}%
    \begin{scope}[xshift=10cm]
        \node[above left] at (0,8) {$p_0$};
        \filldraw (0,8) circle[radius=1.5pt];
        \node[below left] at (2,4) {$p_1$};
        \filldraw (2,4) circle[radius=1.5pt];
        \node[right] at (3, 5) {$\tilde{p}$};
        \filldraw (3,5) circle[radius=1.5pt];
        \draw[->] (0,8)--(2,4);
        \node at (0, 6) {$\gamma$};
        \draw[dotted, ->] (0,8)--(3,5);
        \draw[dotted, ->] (3,5)--(2,4);
    \end{scope}
    \end{tikzpicture}
    \label{fig:decompose}
    }\hspace{1cm}
\subfloat[We can reduce paths to one of these two forms without increasing their $w$-length.]{
    \centering
        \begin{tikzpicture}[>=stealth, thick, scale=0.3]
    \begin{scope}
        \node[above left] at (1,2) {$p_0$};
        \filldraw (1,2) circle[radius=1.5pt];
        \draw[->] (1,2)--(2,1);
        \node[above left] at (4,9) {$p_1$};
        \filldraw (4,9) circle[radius=1.5pt];
        \draw[->] (2,1)--(7, 6);
        \draw[->] (7, 6)--(4,9);
        \draw[dotted, |<->|] (0.8, 1.8)--(1.8,0.8);
        \node at (0.45, 0.45) {$q$};
    \end{scope}%
    \begin{scope}[xshift=20cm, yshift=10cm]
        \node[above left] at (-2,-1) {$p_0$};
        \filldraw (-2,-1) circle[radius=1.5pt];
        \draw[->] (-2,-1)--(-1,-2);
        \node[above left] at (-9,-4) {$p_1$};
        \filldraw (-9,-4) circle[radius=1.5pt];
        \draw[->] (-1,-2)--(-6, -7);
        \draw[->] (-6, -7)--(-9,-4);
    \end{scope}
    \end{tikzpicture}
    \label{fig:classicform}
}
    \caption{Visual representation of some parts of the proof for Proposition~\ref{prop:geo_std_len}.}
\end{figure}

\bibliographystyle{plain}
\bibliography{refs}

\end{document}